\documentclass[10pt,leqno]{amsart}

\usepackage{amsmath, amssymb, amsthm, stmaryrd,tikz,marginnote}
\usepackage{euscript}
\usepackage{verbatim}
\usepackage{graphicx}
\usepackage{amsmath, verbatim}
\usepackage{amssymb,amsfonts,mathrsfs}
\usepackage[colorlinks=true, linkcolor=blue,citecolor=blue, bookmarks=true, pdfstartview=FitH]{hyperref}
\usepackage[driver=pdftex,heightrounded=true,centering]{geometry}
\usepackage{marginnote}
\usepackage{amsthm}
\usepackage{dirtytalk}

\newcommand{\bR}{\mathbb{R}}

\newcommand{\bN}{\mathbb{N}}

\newcommand{\vep}{\varepsilon}

\newcommand{\res}{\mathbf{res}}
\newcommand{\spec}{\mathbf{spec}}
\newcommand{\scl}{\mathrm{sc}}
\newcommand{\ff}{\mathrm{ff}}
\newcommand{\calA}{\mathcal A}
\newcommand{\calR}{\mathcal R}
\newcommand{\calV}{\mathcal V}
\newcommand{\phg}{\mathrm{phg}}
\newcommand{\unif}{\mathrm{unif}}
\newcommand{\olH}{\overline{H}}

\newcommand{\gtil}{\widetilde{g}}

\newcommand{\gbar}{\overline{g}}

\newcommand{\mx}{X}

\renewcommand{\Re}{\mathrm{Re} \;}

\newcommand{\calC}{\mathcal C}
\newcommand{\calL}{\mathcal L}
\newcommand{\calO}{\mathcal O}
\newcommand{\del}{\partial}
\newcommand{\scholder}{\calC_\vep}
\newcommand{\RR}{\mathbb R}
\newcommand{\CC}{\mathbb C}
\newcommand{\frakw}{\mathfrak w}

\newtheorem{theorem}{Theorem}[section]

\newtheorem{prop}[theorem]{Proposition}
\newtheorem{proposition}[theorem]{Proposition}
\newtheorem{lemma}[theorem]{Lemma}
\newtheorem{remark}[theorem]{Remark}
\newtheorem{corollary}[theorem]{Corollary}
\newtheorem{definition}[theorem]{Definition}
\newtheorem{maintheorem}{Theorem}

\numberwithin{equation}{section}
\reversemarginpar 

\begin{document}
\title[Wellposedness of nonlinear flows]{Wellposedness of nonlinear flows \\ 
on manifolds of bounded geometry} 
\begin{abstract}
We present simple conditions which ensure that a strongly elliptic operator $L$ generates an analytic semigroup on 
H\"older spaces on an arbitrary complete manifold of bounded geometry. This is done by establishing the equivalent property 
that $L$ is ``sectorial’’, a condition that specifies the decay of the resolvent $(\lambda  I  - L)^{-1}$ as $\lambda$ diverges from 
the H\"older spectrum of $L$.  As one step, we prove existence of this resolvent if $\lambda$ is sufficiently large, and on this general class 
of manifolds, use a geometric microlocal version of the semiclassical pseudodifferential calculus. The properties of $L$ and $e^{-tL}$ 
we obtain can then be used to prove wellposedness of a wide class of nonlinear flows. We illustrate this by proving wellposedness on 
H\"older spaces of the flow associated to the ambient obstruction tensor on complete manifolds of bounded geometry.
\end{abstract}

\author[Bahuaud]{Eric Bahuaud}
\address{Department of Mathematics,
	Seattle University,
	Seattle, WA, 98122, USA}
\email{bahuaude (at) seattleu.edu}
\author[Guenther]{Christine Guenther}
\address{Department of Mathematics,
	Pacific University,
	Forest Grove, OR 97116, USA}
\email{guenther (at) pacific.edu}
\author[Isenberg]{James Isenberg}
\address{Department of Mathematics,
	University of Oregon,
	Eugene, OR 97403-1222 USA}
\email{isenberg (at) uoregon.edu}
\author[Mazzeo]{Rafe Mazzeo}
\address{Department of Mathematics,
Stanford University,
Stanford, CA, 94305}
\email{rmazzeo (at) stanford.edu}

\date{\today}
\maketitle

\section{Introduction}

Methods from semigroup theory provide an elegant abstract method to establish wellposedness, stability and convergence results for large 
classes of evolutionary 
partial differential equations, such as those which govern geometric heat flows.  For such purposes, the most useful semigroups are those which admit a holomorphic continuation in the `time' variable.  There is a nice
characterization of these {\it analytic semigroups} in terms of their infinitesimal generators. More specifically, write the semigroup as $e^{-tA}$,
where $A$ is a closed operator acting on a certain Banach space $X$. A classical theorem states that the function $t \mapsto e^{-tA}$
from $\RR^+$ to the space $\mathcal L(X)$ of bounded operators on $X$ admits a holomorphic extension to a neighbourhood of $\RR^+$ in $\CC$ 
if and only if the operator $A$ satisfies a property known as {\it sectoriality}, which as indicated below, involves restrictions on the spectrum of A, along with an estimate on its resolvent $(\lambda I - A)^{-1}$. We
refer to \cite[Chapters 5 and 6]{Angenent} for an elementary introduction to sectoriality and analytic semigroups.

In this work, we consider a broad class of geometric operators defined on complete Riemannian manifolds with bounded geometry and show 
that they are indeed sectorial. This is the content of our main result -- Theorem \ref{thm:main-A} -- which is proven in Section \ref{sec:proofThm} 
using tools from microlocal analysis. However, the main consequence of our work is that if the operators corresponding to geometric heat flows 
on manifolds of bounded geometry act on the naturally associated (little) H\"older spaces, and if their symbols satisfy readily verified algebraic 
conditions, then the initial value problems for those flows are well posed, and have good stability and convergence properties. 

More precisely, in applications to PDE, the generator $A$ is a (typically elliptic) differential operator and the sectoriality of such operators is 
known in a variety of settings. Our first goal in this paper is to prove the sectoriality estimate for strongly elliptic differential operators which 
satisfy a certain uniformity property, acting between sections of vector bundles over a complete Riemannian manifold of bounded geometry. We characterize such operators as \emph{admissible}.
The Banach spaces on which we let these operators act are little H\"older spaces; our emphasis on these is because of their use in 
applications to nonlinear problems. The immediate examples of such operators are generalized Laplacians, i.e., operators of the form 
$\nabla^* \nabla + \mathcal R$, where $\mathcal R$ is an endomorphism usually constructed from the curvature tensor of the underlying metric 
and its covariant derivatives. However, the method of proof extends naturally to allow us to prove this estimate for more general higher order 
operators as well.

Our second goal is to apply this sectoriality to deduce stability estimates for nonlinear parabolic evolution equations on these manifolds, as
before, acting on little H\"older spaces. 
We are particularly interested in geometric curvature flows, e.g., the Ricci flow, the mean curvature flow, and certain relatively unexplored higher
order flows such as the one associated to the ambient obstruction tensor (see \S \ref{sec:applicats} for a description of this). As discussed below, these flows typically 
require some sort of gauge fixing in order to become suitably parabolic. The applications  in this paper illustrate how one can easily establish wellposedness of quite general flows, on spaces that are not necessarily compact, using this sectoriality property. In a subsequent paper we describe an application in which sectoriality is a key part of the proof of a long-time stability and convergence result.

Sectoriality for admissible operators on manifolds on spaces of uniformly bounded geometry has, in fact, been treated previously,
notably by H. Amann and his collaborators, see for example \cite{Amann2}, \cite{ES}.  Those techniques are spread out over several papers, considerably more abstract 
and, from a geometric point of view, perhaps less accessible.  Our goal here is to provide a straightforward and
hopefully more approachable proof which should be more convenient for geometric applications. 

Let us briefly recall the functional analytic setting in more detail.  The fundamental idea when applying semigroup theory to differential equations is 
to recast the problem as an ordinary differential equation with values in some Banach space.  Let $X$ be a complex Banach space and $D$ a dense linear subspace.
Consider the $X$-valued autonomous ordinary differential equation
\begin{align} \label{eqn:ODEinBanach}
\frac{d u}{dt} = F( u(t ) ),
\end{align}
where $u: [0,T) \to X$ is a $\calC^1$ mapping.  Here $F: D \to X$ is a (nonlinear) Fr\'echet differentiable map satisfying certain structural assumptions.  
The most important of these is that the linearization $L$ of $F$ at $u_0 \in D$ is a sectorial operator on $X$. We now explain this hypothesis.

The resolvent set of a closed linear operator $L: X \to X$ is the subset $\res_X(L) \subset \CC$ consisting of all numbers $\lambda$ such that 
$(\lambda I - L): D \to X$ has an inverse which is a bounded operator on $X$.  $L$ is called {\it sectorial} on $X$ if it satisfies: 
\begin{enumerate} 
\item [\underline{S1}.] The resolvent set $\res_{X}(L)$ contains a sector of opening angle $2\theta < \pi$ which contains a
left half-plane $\mathrm{Re}\,\lambda \leq \omega$ for some $\omega \in \bR$, i.e., there exists $\theta \in (0,\pi/2)$ such that
\[ 
\res_{\mx}(L) \supset S_{\omega, \theta} := \{ \lambda \in \CC \setminus \{\omega\}: |\arg(\omega -\lambda) | > \theta\},  
\quad  \mathrm{and}
\]
	\item [\underline{S2.}] There exists a constant $C > 0$ so that for all $\lambda \in S_{\omega, \theta}$, we have
	\begin{equation*}
	\|(\lambda I - L)^{-1} \|_{\mathcal{L}(\mx)} \leq \frac{C}{|\lambda - \omega|}.	\end{equation*}
\end{enumerate}
These conditions on $L$ turn out to be equivalent to the analyticity of the semigroup $e^{-tL}$, and from this a wealth of  wellposedness and 
regularity results are then available.  Many nonlinear parabolic partial differential operators can be cast into this framework, see \cite{Lunardi} 
for a thorough account of this general theory.  Of particular importance is that spectral stability analysis of $L$ yields stability results for the 
nonlinear problem in some cases.

As for the geometric applications, recall that a complete Riemannian manifold $(M,g)$ is said to have bounded geometry (of a certain order) if 
its injectivity radius is bounded from below and there is a uniform bound for the norm of the curvature tensor and its covariant derivatives up 
to some order. This is equivalent to uniform control of the coefficients of the metric and its inverse in an atlas of normal coordinate balls of 
uniform radius. This class of spaces includes compact Riemannian manifolds, of course, but also complete noncompact manifolds which are 
asymptotically Euclidean, conic, cylindrical 
or hyperbolic, respectively, or more generally which are asymptotically modeled on other noncompact symmetric or homogeneous spaces. 

On any such space $(M,g)$ we consider elliptic differential operators, acting between sections of vector bundles, which satisfy uniformity conditions
on their coefficients in these local uniform coordinate charts.  The most obvious examples are operators determined directly from the metric $g$,
for example, generalized Laplace-type operators
\[
L = \nabla^* \nabla + \mathcal R,
\]
acting on sections of some tensor bundle over $M$. Here $\nabla$ is the induced covariant derivative on this bundle, $\nabla^*$ its adjoint,
and $\mathcal R$ a symmetric endomorphism built out of tensor products of contractions of the curvature tensor and its covariant derivatives.  
Many aspects of the mapping properties of $L$ on $L^2(M, dV_g)$ can be deduced from Hilbert space techniques.  However, it is often
more convenient for nonlinear geometric problems to consider this operator acting on weighted H\"older spaces instead.  Consider a general
weighted H\"older space $X = \mathfrak{w} \, \mathcal{C}^{k,\alpha}(M; E)$, where $\mathfrak{w}$ is a (strictly positive) weight function.  
The assumptions imposed on $\mathfrak{w}$ are specified in Definition~\ref{wthyp}.   More generally, we consider $L$ to be a strongly 
elliptic operator which satisfies certain uniformity conditions specified in Definition~\ref{admop}, acting on weighted H\"older sections of some 
vector bundle. Note that by replacing $L$ with $\mathfrak w^{-1} L \mathfrak w$
we may as well consider $L$ as acting simply on $\calC^{k,\alpha}(M;E)$.  The conditions on $\mathfrak w$ are precisely the ones
necessary for this conjugated operator to satisfy the same uniformity hypothesis. 

The first main result of this paper states that any uniform strongly elliptic operator as above is {\it sectorial}.
\begin{maintheorem} \label{thm:main-A}
Let $(M^n, g)$ be a complete Riemannian manifold of bounded geometry of order $\ell + \alpha' > m+k + \alpha$, where $m \in \mathbb N$,
$k \in \mathbb N_0 = \mathbb N \cup \{0\}$ and $0 < \alpha < \alpha' < 1$, and suppose that $L$ is an admissible elliptic operator of 
order $m$, i.e., it is strongly elliptic satisfying the uniformity hypotheses in Definition~\ref{admop2} below, which we let act on $X = 
\calC^{k,\alpha}(M;E)$ for some bundle $E$ over $M$. Then $L$ is sectorial on $X$.
\end{maintheorem}

We briefly enumerate the main ideas in the proof.  We begin with the observation that the sectoriality estimate is equivalent to a uniform estimate for the 
associated {\it semiclassical operator} $\zeta I - \vep^m L$, where $\vep = |\lambda|^{-1/m}$ and $\zeta = \lambda/|\lambda|$.   
The first step is to show that this operator 
is actually invertible on $X$ when $\vep$ is sufficiently small. This is deduced by constructing an approximation for the inverse of this 
operator, which is called the \emph{semiclassical resolvent}. This involves a detour into the methods of geometric microlocal analysis, and the 
construction itself is sketched in some detail in Section \ref{scpc} in order to be as self-contained as possible. (This geometric microlocal analytic
construction has appeared implicitly in the literature before, but does not seem to appear explicitly in a readily available form elsewhere.)
Having established the existence of $(\zeta I - \vep^m L)^{-1}$ as a bounded operator on $X = \calC^{0,\alpha}(M)$ for each $\vep > 0$ 
sufficiently small, we need to establish uniformity of its norm.  This is argued by contradiction: we show using a number of rescaling and blowup
arguments that the failure of uniformity of this estimate would lead to various impossible conclusions.  A key feature of this argument
is that we parlay the (essentially tautological) uniform estimates of this operator acting on {\it semiclassical} H\"older spaces (see \S \ref{sec:background}) 
to uniformity for the action of this operator on standard H\"older spaces.  The passage from sectoriality on $\calC^{0,\alpha}$ to sectoriality
on $\calC^{k,\alpha}$ is a trivial extension. 

The key motivation for all this work is its application to proving wellposedness of geometric flows on complete noncompact manifolds. We obtain the following theorem: 
\begin{maintheorem} \label{thm:main-B} Let $(M,g)$ be a complete Riemannian manifold of bounded geometry of order $k+m+\alpha'$,  where $m \in \mathbb N$,
$k \in \mathbb N_0 = \mathbb N \cup \{0\}$ and $0 < \alpha < \alpha' < 1$, let $U$ be an open subset of $\calC^{m+k,\alpha}$, and let $F: U \to \calC^{k,\alpha}$ be a smooth elliptic operator of order $m$ such that the linearization $DF_u$ at any $u \in U$ is admissible.  Then for any $u_0 \in U$, there exists $T > 0$ so that the equation
\[ \frac{du}{dt} = F( u(t) ), \; u(0) = u_0 \]
has a unique solution $u: [0,T) \to \calC^{k,\alpha}$.  Moreover, any two solutions $v(t)$ and $w(t)$ with initial values $v_0$ and $w_0$ in $U$ satisfy
\[  \|v(t) - w(t)\|_{\calC^{m+k,\alpha}} \leq C \| v_0 - w_0\|_{\calC^{m+k,\alpha}}, \; \; \mbox{for all} \; t \in [0,T) \]
\end{maintheorem}

There are many possible applications. We illustrate this by focusing on a flow for a metric involving the ambient obstruction tensor, $\calO_n$, developed by Fefferman and Graham \cite{FG}.  As we review in \S \ref{sec:applicats}, this is a conformally invariant tensor that involves $n$ derivatives of the metric.  Due to the higher-order nature of the system of equations for this flow, the usual technique of using an exhaustion and maximum principles to prove existence are not easy to apply.  To our knowledge our method is the only wellposedness result to date for this flow.  Our main result is:

\begin{maintheorem} \label{thm:main-C}
Let $(M^n,g)$ be a complete Riemannian manifold of bounded geometry of order $2n + \alpha'$, with even dimension 
$n = 2\ell$, and where $0 < \alpha < \alpha' < 1$. If $g_0$ is any smooth metric on $M$, then there exists $T > 0$ and a family of unique metrics $g: [0,T) \to \calC^{n,\alpha}(M,g)$ solving the ambient obstruction flow 
\begin{align} 
\begin{cases} \partial_t g &= \mathcal{O}_n(g) + c_n (-1)^{\frac{n}{2}} ( (-\Delta)^{\frac{n}{2}- 1} S ) g \\
g(0) &= g_0,
\end{cases}
\end{align}
where $c_n = (2^{n/2 - 1} ( \frac{n}{2} - 2)! (n-2) (n-1))^{-1}$ and $S$ is the scalar curvature of $g$.
\end{maintheorem}

The remainder of this paper is structured as follows.  In \S \ref{sec:background} we describe the analytic and geometric background.  After discussing sectoriality, we define manifolds of bounded geometry and define the operators of interest.  We discuss pointed limits of manifolds of bounded geometry, and prove a relationship between the resolvent set of an operator and its limiting operators under this construction,   Proposition \ref{limspecrel} that may be of independent interest.  In \S \ref{sec:proofThm} we explain the reduction of sectoriality to semiclassical estimates and prove Theorem \ref{thm:main-A}.  This section presumes the existence of uniform bounds for the semiclassical resolvent of an admissible operator, and we give a detailed construction of this resolvent in \S \ref{scpc} using the techniques of geometric microlocal analysis.  Finally in \S \ref{sec:applicats} we apply our results to prove Theorem \ref{thm:main-B}, and conclude with the proof of the wellposedness result for the ambient obstruction flow, Theorem \ref{thm:main-C}.

\subsubsection*{Acknowledgments} The authors thank Jack Lee, Yoshihiko Matsumoto, Andr\'as Vasy, and Guofang Wei for useful conversations during this work.  
This work was supported by collaboration grants from the Simons Foundation (\#426628, E. Bahuaud and \#283083, C. Guenther).  
J. Isenberg was supported by NSF grant PHY-1707427. 

\section{Background}
\label{sec:background}

\subsection{Sectoriality} \label{sec:analytic-bkgd}

We begin by defining what it means for a closed unbounded operator acting on a Banach space to be sectorial. The abstract notion
of sectoriality, and its precise relationship with the theory of analytic semigroups is classical and can be found, for example, in \cite{Yosida}[Chapter IX].
The monographs \cite{Amann1, Lunardi} contain applications of sectoriality to the study of evolution equations, and the papers \cite{BGI, GIK} focus on
its specific application to Ricci flow. 

Let $X$ be a complex Banach space and $\mathcal{L}(X)$ the space of bounded linear operators on $X$; we denote the operator norm by
$\| \cdot \|_{\mathcal{L}(X)}$.  Suppose that $L$ is a closed {\it unbounded} linear operator on $X$ which has dense domain 
$D \hookrightarrow X$. The \emph{resolvent set} of $L$, $\res_{X}(L)$,  is the set of $\lambda \in \CC$ for which the resolvent operator
\[
R_L(\lambda) := (\lambda I - L)^{-1}
\]
lies in $\mathcal{L}(X)$.  The range of $R_L(\lambda)$ is the 
domain $D$. The \emph{spectrum} of $L$, denoted $\spec_X(L)$, is the complement 
$\CC \setminus \res_{X}(L)$. 

\begin{definition}
\label{sectorialdef}
A closed unbounded operator $L:  X \to X$ with domain $D$ is sectorial in $X$ if: 
\begin{enumerate} \label{def-sec}
\item [\underline{S1}.] The resolvent set $\res_{X}(L)$ contains a sector of opening angle $2\theta < \pi$ which contains a
left half-plane $\mathrm{Re}\,\lambda \leq \omega$ for some $\omega \in \bR$, i.e., there exists $\theta \in (0,\pi/2)$ such that
\[ 
\res_{\mx}(L) \supset S_{\omega, \theta} := \{ \lambda \in \CC \setminus \{\omega\}: |\arg(\omega -\lambda) | > \theta\},  
\quad  \mathrm{and}
\]
\item [\underline{S2}.] There exists a constant $C > 0$ so that 
\begin{equation} 
\label{ResEst}
\| R_L(\lambda) \|_{\mathcal{L}(\mx)} \leq \frac{C}{|\lambda - \omega|}\ \ \mbox{for all}\ \ \lambda \in S_{\omega,\theta}.
\end{equation}
\end{enumerate}
We often simply say that $L$ is sectorial if the space $X$ is understood. 
\end{definition}
\begin{remark}
We adopt the convention that the spectrum of $L$ lies in a sector with an acute opening angle and is strictly
contained in a {\it right} half-plane $\mathrm{Re}\, \lambda \geq \omega$.  In the applications below, 
$L$ is a differential operator with the leading part equal to some power of an iterated Laplacian $\Delta^k$, and our convention then 
agrees with the one where the $L^2$ spectrum of $\Delta$ lies in the positive half-line.  Note that our convention is different from our earlier work \cite{BGI} and the monograph \cite{Lunardi}.
\end{remark}

Sectoriality is equivalent to an apparently weaker condition: 
\begin{lemma}[Proposition 2.1.11 of \cite{Lunardi}]\label{lem:halfplane} Let $X$ be a complex Banach space, and 
$L: X \to X$ a closed linear operator with dense domain $D$ such that $\res_X(L)$ contains a closed half-plane  $\{\lambda \in \CC: 
\Re \lambda \leq \omega\}$, for some $\omega \in \RR$. If there exists a constant $C > 0$ such that 
\begin{equation} 
\| \lambda (\lambda - L)^{-1} \|_{\mathcal{L}(\mx)} \leq C, 
\label{sectest}
\end{equation}
for all $\lambda$ in this half-plane, then $L$ is sectorial.
\end{lemma}
\begin{proof}
By \eqref{sectest}, $||R_L(\omega + i\mu)||_{\calL(X)} \leq \frac{C}{|\omega + i\mu|}$, so if 
$|\lambda - (\omega + i\mu)| \leq |\omega + i\mu|/2C$, then
\[
\begin{split}
(\lambda I - L) = & ((\omega + i\mu) I - L)  + (\lambda - (\omega + i\mu)) I  \\= 
&((\omega + i\mu)I  - L) \left( I + \left(\lambda - \left(\omega + i\mu\right) \right) R_L(\omega + i\mu) \right). 
\end{split}
\]
The second factor on the right is of the form $I + A$ where $||A|| \leq 1/2$, and the first factor on the right is invertible by 
hypothesis, so their product is invertible and hence $\lambda \in \res_X(L)$.    Since the radii of the balls around $\omega + i\mu$ on which 
the resolvent is defined grow asymptotically linearly in $\mu$, the union of the original half-plane together with these 
balls contains a sector $S_{\omega,\theta}$ for some $\theta \in (0,\pi/2)$. The estimate \eqref{ResEst} follows. 
\end{proof}

\subsection{Manifolds of bounded geometry} \label{sec:geo-bkgd}
As stated in the introduction, we consider the sectoriality of a general class of \emph{admissible} elliptic differential operators of even 
order $m = 2m'$, $m' \in \mathbb{N}$ acting between H\"older spaces, $L: \calC^{m+k,\alpha}(M,g)~\to~\calC^{k,\alpha}(M,g)$, where $(M,g)$ 
is a complete manifold with bounded geometry of order at least $m+ k + \alpha'$ for some $\alpha' \in (\alpha, 1)$. We may consider
any such $L$ as an {\it unbounded} operator on $\calC^{k,\alpha}(M,g)$. 

In this paper we work exclusively with the `little' H\"older spaces, which by definition are the closure of $\calC^\infty$ in
the corresponding H\"older norm. For a given $k, \alpha$, the little H\"older space of this order is a separable closed subspace
of the full H\"older space.  To lighten the notational burden, we denote this little space by the same symbol $\calC^{k,\alpha}$,
with the understanding that we never use the big H\"older spaces here. 

In this section we begin with a description of manifolds of bounded geometry, and direct the reader to Subsections \ref{subsec:funct-spcs} and \ref{sec-admiss} for more detail on the operators and function spaces which appear below. Briefly, key examples of the operators we consider
are elliptic operators arising naturally in geometric analysis of the form
\[
L = (\nabla^* \nabla)^{m'} + \text{lower order terms}
\]
where $\nabla$ is the covariant derivative acting on sections of some Hermitian vector bundle $V$ over $M$, and where
the lower order terms involve the curvature tensor of the underlying metric.  More generally, we also consider such operators 
acting between weighted (little) H\"older spaces: 
\begin{equation}
L: \, \frakw \, \calC^{m+k,\alpha}(M,g) \longrightarrow \frakw \, \calC^{k,\alpha}(M,g),
\label{wtL}
\end{equation}
where $\frakw$ is a weight function satisfying certain uniformity hypotheses, see Definition \eqref{wthyp}. The mapping \eqref{wtL} is equivalent to 
\[
(\frakw)^{-1} L \frakw = (\nabla^* \nabla)^{m'}  + S: \calC^{m +k ,\alpha}(M,g) \longrightarrow \calC^{k,\alpha}(M,g), 
\]
where $S$ is an operator of order $m-1$ which includes both the conjugate of the lower order terms in $L$ and also $(\frakw)^{-1} [ (\nabla^* \nabla)^{m'}, \frakw]$.

Let us begin by recalling the definition of a manifold of bounded geometry: 
\begin{definition}
A complete Riemannian manifold $(M,g)$ is said to have bounded geometry of order $\ell + \alpha'$, where $\ell \in \mathbb N_0$ and
$0 \leq \alpha' < 1$, if: 
\begin{itemize}
\item[a)] There exists a radius $r_0 > 0$ such that for every $q \in M$, the exponential map 
$\exp_q: \{v \in T_qM: |v| < r_0\} \to B_{r_0}(q)$ is a diffeomorphism, i.e., the injectivity radius at $q$ is greater than $r_0$;
\item[b)] For every $q \in M$, the components of the pulled back metric $\exp_q^* g$ are bounded in $\calC^{\ell,\alpha'}$ and the components 
of the matrix inverse of $\exp_q^*g$ are bounded in $\calC^0$ on $\{ v \in T_q M: |v| < r_0\}$, where the bounds are independent of $q$,
and hence uniform over $M$. 
\end{itemize}
\end{definition}
\begin{remark} We have denoted the fractional part of this uniformity order by $\alpha'$ to distinguish it from the $\alpha$ index in the
H\"older spaces we are using. We need the extra room given by the inequality $\alpha' > \alpha$
when taking limits using the Arzela-Ascoli Theorem. 
\end{remark}

It is often easier to check an intrinsic version of condition b). As discussed in \cite{Eichhorn}, for example, if $\ell \geq 1$ and $\alpha'= 0$, 
then b) is implied by 
\begin{itemize}
\item [b')] $\sup_{j \leq \ell} | \nabla^j \mathrm{Riem} | \leq C_\ell$ for some constant $C_\ell$. 
\end{itemize}
The proof of that implication presumably generalizes without difficulty to the case where $\alpha' \neq 0$. 

Any compact Riemannian manifold has bounded geometry of order equal to the regularity class of the metric. There are many other
natural examples of manifolds with bounded geometry. We list some familiar classes: 
\begin{itemize}
\item[i)] Asymptotically Euclidean or asymptotically conic manifolds,
\item[ii)] Manifolds with asymptotically cylindrical ends,
\item[iii)] Asymptotically (real) hyperbolic manifolds,
\item[iv)] Asymptotically complex hyperbolic manifolds,
\item[v)] Any symmetric space $M = G/K$ of noncompact type, with invariant metric $g$, 
or indeed, any perturbation $g = g_0 + h$ of the symmetric metric $g_0$, where $| \nabla^j h|_{g_0} \leq C_j$ for $j \leq \ell$,
\item[vi)]  Any infinite cover $(M,g)$ of a compact manifold $(M_0, g_0)$.
\end{itemize}

Let us briefly recall each of these classes.  Before doing so, observe that if $(M,g)$ has bounded geometry of order $\ell + \alpha'$
and if $\tilde{g} = g + h$, where $|h|_g \leq 1-\epsilon$ for some $\epsilon \in (0,1)$ (so that $\tilde{g}$ is boundedly equivalent to $g$)
and the $\calC^{\ell,\alpha'}$ norms of the components of $h$ are controlled in $g$ normal coordinate charts, then $\tilde{g}$ also has 
bounded geometry of order $\ell + \alpha'$. This means that we may as well describe these various classes of spaces in their simplest model forms.
Bounded geometry then follows for any metrics which are perturbations of these models in the sense above.  We are particularly 
interested in perturbations which decay to the appropriate model metrics in a suitable sense at infinity, and shall mention the rate of decay
in each of these cases.

\

\noindent{\bf Asymptotically Euclidean and asymptotically conic metrics.}
A Riemannian manifold $(M^n,g)$ is called \emph{conic at infinity} if there exists a compact Riemannian manifold $(Y, h_0)$ of dimension
$n-1$, a compact set $K$ in $M$ and a diffeomorphism from $M \setminus K$ to $[r_0, \infty) \times Y$, such that
\[
g = dr^2 + r^2 h_0.
\]
More generally, $(M,g)$ is called \emph{asymptotically conic} (AC) if it can be written as the sum of a metric which is conic at infinity and an 
extra term $k$ which satisfies $|\nabla^j k|_g  \leq C r^{-\beta - j}$ for some $\beta > 0$ and for $0 \leq j \leq \ell$, and 
$[\nabla^\ell k]_{0,\alpha'} \leq C r^{-\beta - \ell - \alpha'}$.    

In the following examples, we shall simply state a decay rate, e.g. $r^{-\beta}$, but with corresponding decay rates on the derivatives
implicit. 

An AC space is called \emph{asymptotically Euclidean} (AE) if $(Y, h_0)$ is isometric to the standard sphere.  Elliptic theory on this class of 
spaces has been very thoroughly studied for several decades; see \cite{Bartnik} for a survey of results from a `classical' perspective,
and \cite{Melrose-Mendoza} for another approach which appears frequently below. There are important generalizations of AE and AC spaces 
that arise in various geometric settings, including the classes of \emph{quasi-asymptotically conic} (QAC) manifolds \cite{DegMaz}, certain of 
the four-dimensional `gravitational instantons' (of types ALE/F/G/H), along with their higher dimensional generalizations \cite{ChenChen}, etc.

\

\noindent{\bf Asymptotically cylindrical metrics.} If $(M^n, g)$ is cylindrical at infinity, then outside some compact set it is isometric to a product cylinder $(a,\infty)\times Y^{n-1}$
with metric $dt^2 + h_0$. Setting $r = e^{-t}$ we arrive at the equivalent form
\[
g = \frac{dr^2}{r^2} + h_0,
\]
which is conformal to the exact conic metric $dr^2 + r^2 h_0$, which is useful for translating results from one setting to the other.
The allowable perturbations in this setting decay like $e^{-\beta t}$ as $t \to \infty$, or equivalently, like $r^{\beta}$ as $r \to 0$.

\

\noindent{\bf Asymptotically hyperbolic metrics.}
Next, suppose that $M$ is a compact Riemannian manifold with boundary. Fix a smooth boundary defining function
$\rho$ on $M$, i.e., $\rho \geq 0$ on $M$, $\rho^{-1}(0) = \del M$, and $d\rho \neq 0$ at the boundary.  Fix also a metric
$h_0$ on $\del M$. The class of `exact' asymptotically hyperbolic (AH) metrics consists of metrics taking the form
\[
g = \frac{d\rho^2 + h_0}{\rho^2}
\]
near $\rho = 0$. The metric $\gbar = \rho^2 g$ is called a conformal compactification of $g$. Allowable perturbations in this case are tensors
which decay like $\rho^\mu$ for some $\mu > 0$. 

This geometry mimics that of the Poincar\'e ball model of hyperbolic space, where 
\[
g = \frac{ 4 |dz|^2}{ (1-|z|^2)^2}.
\]
Thus $\rho = (1-|z|^2)/2$ and the Euclidean metric $|dz|^2$ equals a particular choice of conformal compactification $\bar{g}$.  Since
$\rho$ is not canonically defined in terms of $g$, only  the conformal class of $\gbar$ (and in particular, $\gbar|_{T \del M}$) is intrinsic to $g$.

To see that an AH space has bounded geometry, a simple calculation shows that the sectional curvatures of any AH metric 
tend to $-1$ and covariant derivatives of the curvature tensor tend to $0$, all as $\rho \to 0$. If $q \in M$ and 
$\rho(q) = \epsilon$, then the $\gbar$ ball $B_{\epsilon/2}(q)$ has inradius and diameter which are uniformly bounded away from both $0$ 
and $\infty$, and the restriction of $g$ to any such ball converges to a hyperbolic metric on a ball
of nonzero radius. We refer to \cite{Lee} for more details on this.   

\

\noindent{\bf Asymptotically complex hyperbolic metrics.}
One generalization of this last example is to the class of asymptotically complex hyperbolic manifolds.
There are various ways to define these spaces; we refer to \cite{Biquard} for one approach and a more extended discussion
than the one below. Proceeding as in the AH case, let $M$ be a compact Riemannian manifold with boundary, of even real dimension $2n$, and $\rho$ a boundary defining function. Suppose that $\eta$ is a contact form on $\del M$, i.e., $\eta$ is a $1$-form such that $\eta \wedge (d\eta)^{n-1}$ 
is everywhere nonvanishing.  Let $T$ denote the Reeb vector field on $\del M$, i.e., the unique vector
field such that $\eta(T) \equiv 1$ and $d\eta(T, \cdot) \equiv 0$.  Finally, choose a set of smooth
independent vector fields $X_1, \ldots, X_{2n-2}$ which span the kernel of $\eta$ in $T \del M$, all on $\del M$.
Let $\eta, \omega_1, \ldots, \omega_{2n-2}$ be the coframe dual to $T,  X_1, \ldots, X_{2n-2}$.
Fixing a product decomposition of a collar neighborhood of $\del M$ in $M$, we say that $g$ is (exact) 
asymptotically complex hyperbolic (ACH) if 
\[
g = \frac{d\rho^2 + \sum \omega_j^2}{\rho^2} + \frac{\eta^2}{\rho^4}
\]
in that neighborhood. (Here $\omega_j^2 = \omega_j \otimes \omega_j$ and $\eta^2 = \eta \otimes \eta$.)
The difference with the AH case is of course simply that the metric blows up faster in the $\alpha^2$ direction.
A CR (Cauchy-Riemann) structure involves not only the hyperplane bundle $\mathrm{ker}\, \eta$ but also an endomorphism
$J$ on this subbundle which satisfies $J^2 = -I$; however, this almost complex structure is not relevant
to these metric asymptotics. An allowable perturbation $k$ again decays like some $\rho^\mu$ as $\rho \to 0$.  

This mimics a standard representation of the complex hyperbolic metric on the unit ball in $\CC^n$ (with holomorphic 
sectional curvature $-4$), 
\[
g = \frac{ g_{\mathrm{Euc}} }{1-r^2} + \frac{r^2 \left( dr^2 + ( J dr )^2 \right) }{(1-r^2)^2},
\]
where $g_{\mathrm{Euc}}$ is the Euclidean metric, $r = |z|$, and $z \in \CC^n$. The monograph \cite{Biquard}  
explains the relationship of this construction to CR geometry of the boundary. Bounded geometry of ACH metrics can be 
proved quite similarly to the AH case.  The `cubes' of
approximate radius $1$ have (approximate) dimensions $\epsilon$ in the $\del_\rho$ and $X_j$ directions
and $\epsilon^2$ in the $T$ direction.

There are further generalizations to classes of exact and asymptotically quaternion hyperbolic metrics and
(asymptotically) octonion hyperbolic planes. These involve generalizing the contact structures used to define 
the ACH metric; see \cite{Biquard}. 

\

\noindent{\bf Other examples.} A Riemannian symmetric space $M = G/K$ of noncompact type, or more generally a Riemannian 
homogeneous space $M = G/H$ with invariant metric, again has bounded geometry. The definitions are
a bit more intricate, and we point to one of the many standard references on the subject, for example \cite{Hel},
at least for the symmetric space case. 

Infinite covers of compact Riemannian manifolds, where the metric is obtained by pullback via the covering map, 
are an interesting class of spaces. General results about the $L^2$-resolvent of even the scalar Laplacian 
on such manifolds are almost nonexistent, except if the covering group is `small' (e.g., amenable). 
Perhaps surprisingly, we are able to carry out the analysis below for the resolvents of admissible
operators on $\calC^{0,\alpha}$, not only on these classes of spaces, but even on general manifolds of bounded geometry. 

\

\noindent{\bf Pointed limits of manifolds of bounded geometry.} To conclude this subsection, we recall an important construction in the category of manifolds with bounded geometry 
which shows that this class of spaces is complete in a certain sense.  Let $(M,g)$ have bounded geometry of some order 
$\ell + \alpha'$ and consider the sequence of pointed spaces $(M, g, p_j)$ where $p_j$ is a sequence of points 
in $M$ which diverges to infinity. Then there is a complete Riemannian manifold $(M_\infty, g_\infty,p_\infty)$,
which is the pointed Gromov-Hausdorff limit of the sequence $(M,g,p_j)$. More specifically, for any 
$R > 0$, the $g$-ball of radius $R$ around $p_j$ in $M$ converges, at least up to passing to 
a subsequence, as a Riemannian space, to the $g_\infty$ ball of radius $R$ around $p_\infty$.
This convergence can be shown to occur in $\calC^{\ell, \alpha''}$ for any $0 < \alpha'' < \alpha'$.

This is a standard fact in the `convergence theory' of Riemannian manifolds.  This type
of construction is originally due to Cheeger, and admits many generalizations. A version that 
encompasses the particular statement above appears as Theorem 11.36 in Petersen's book 
\cite{Petersen2016}. (We are grateful to Guofang Wei for pointing out this reference.) 

As a very brief sketch of how this is proved, one first shows that some subsequence of
the $(M, g, p_j)$ converges in the Gromov-Hausdorff topology to $(M_\infty, g_\infty, p_\infty)$;
this topology is, of course, quite weak and in fact one initially only obtains its metric space structure.
Further analysis shows that this convergence happens in a much stronger topology.  Indeed,
using the bounds on the metric tensor in normal coordinates, we may extract a subsequence of 
metrics on each ball $B_g(p_j, c)$ which converges in $\calC^{\ell,\alpha''}$.  The curvature bounds 
imply that any ball of larger radius $R$ can be covered by a controlled number of balls of radius 
$r_0$. By a successive diagonalization argument, the metric tensor converges on each one of these.
The lower bound on the injectivity radius is used to show that there is no collapsing in the limit. 

Here are two examples of this sort of convergence. If $(M,g)$ has an asymptotically cylindrical end, and if $p_j$ diverges
along this cylindrical end, then the corresponding limit space $M_\infty$ is the Riemannian product cylinder $\RR \times Y$. 
Similarly, if $(M,g)$ is asymptotically hyperbolic, and $p_j$ diverges to some point $\bar{p}$ on the boundary of
the conformal compactification of $M$, then $M_\infty$ is a copy of hyperbolic space $\mathbb H^n$.  These illustrate
that the limit space can `lose' a lot of the topology of the original manifold $M$. 
\subsection{Function spaces}
\label{subsec:funct-spcs}
The definition of bounded geometry of order $\ell  + \alpha'$ relies on the H\"older norms in uniform local coordinate
charts, and the standard local Euclidean definition can be used.  To define H\"older spaces globally on $M$ we need
to say a bit more.

Fix a manifold $(M,g)$ with bounded geometry of order $\ell +\alpha'$, We define the H\"older spaces $\calC^{\kappa,\alpha}(M,g)$ for any 
$0 \leq \kappa \leq \ell$ and $0 < \alpha < \alpha' < 1$.  
Introduce the $\calC^{\kappa,\alpha}$ norm
\[
||u||_{\kappa,\alpha} := \sum_{j = 0}^\kappa \sup |\nabla^j u|_g + \sup_B \sup_{x, y \in B \atop x \neq y} \frac{ |\nabla^\kappa u(x) - \nabla^\kappa u(y)|}{
\mathrm{dist}_g(x,y)^\alpha}.
\]
The supremum in the final term on the right is over all geodesic balls $B \subset M$ of radius $r_0$ (as given in
the definition of bounded geometry).  We assume that the tensor bundles over each such $B$ are trivialized,
for example using the exponential map based at the centers of these balls.  If we are considering sections of some other vector
bundle $V \to M$, we assume the existence of a uniform set of local trivializations, relative to some `uniform' cover
of $M$ by balls $B_{r_0}(q_j)$. The details are straightforward and left to the reader. 

As noted in the Introduction, in this paper we use the \emph{`little' H\"older spaces}, $\calC^{\kappa,\alpha}(M,g)$ exclusively; 
by definition, these are the completion of $\calC^\infty$ with respect to the norms above. These little H\"older spaces have several 
of advantages: they are separable, and it is possible to use approximation arguments with them; further, one can easily define interpolation spaces that allow access to maximal regularity theory for nonlinear applications (Section Ch 35 in \cite{RFIV}).  

We recall the simple and useful fact that an equivalent norm is obtained by taking the supremum over all $x \neq y$ in $M$
in the final H\"older seminorm, rather than just over $x, y \in B$; in other words, we claim that
\[
\sup_{x, y \in M \atop x \neq y} \frac{ |\nabla^\kappa u(x) - \nabla^\kappa u(y)|}{\mathrm{dist}_g(x,y)^\alpha} \leq C ||u||_{\kappa,\alpha}
\]
for some fixed $C > 0$. If $\mathrm{dist}_g(x,y) < r_0$, this is obvious, while if $\mathrm{dist}_g(x,y) \geq r_0$, then
\[
\frac{ |\nabla^\kappa u(x) - \nabla^\kappa u(y)|}{\mathrm{dist}_g(x,y)^\alpha} \leq 2 r_0^{-\alpha} \sup |\nabla^\kappa u|_g.
\]
It is also clear that for each $B$,
\[
\sup_{x,y \in B \atop x \neq y} \frac{ |\nabla^\kappa u(x) - \nabla^\kappa u(y)|}{\mathrm{dist}_g(x,y)^\alpha} 
\leq \sup_{x,y \in M \atop x \neq y} \frac{ |\nabla^\kappa u(x) - \nabla^\kappa u(y)|}{\mathrm{dist}_g(x,y)^\alpha}.
\]
Hence, taking the supremum over all balls $B$ on the left, we conclude that we may define
the $\calC^{\kappa,\alpha}$ seminorm either as we have done initially, or else by replacing the final term in that
definition with one where the supremum is taken over all distinct $x, y \in M$. 

It is clear that it  makes sense to consider the spaces $\calC^{\kappa,\alpha}$ on a manifold of bounded geometry
of order $\ell + \alpha'$ only when $\kappa \leq \ell$ and if $\kappa = \ell$ then $\alpha \leq \alpha'$.  
We assume the strict inequality $\alpha < \alpha'$ because, as described in the last subsection, the pointed limit
of spaces of order $\ell + \alpha'$ may only have order $\ell + \alpha''$ for any $\alpha'' < \alpha'$.  (Actually,
by a standard real analysis argument, the limiting space does have order $\ell + \alpha'$ but the convergence
only takes place in the weaker norm, which may be important at certain points.)

We also define the family of {\it semiclassical} H\"older spaces $\calC^{\kappa,\alpha}_\vep(M,g)$, where $\vep$ is a parameter
in $(0,1]$. The name comes from their natural association with families of operators undergoing semiclassical degeneration,
which is described below. These spaces appear in a fundamental way in the arguments of Section \ref{sec:proofThm}. 
For any given $\kappa \leq \ell$ and $\alpha$ (as usual with $\alpha \leq \alpha'$ if $\kappa = \ell$), 
this family of spaces contains all functions $u$ such that  
\[
||u||_{\kappa, \alpha, \vep} :=  \sum_{j = 0}^\kappa  \vep^j \sup |\nabla^j u|_g + 
\vep^{\kappa + \alpha} \sup_B \sup_{x, y \in B \atop x \neq y} \frac{ |\nabla^\kappa u(x) - \nabla^\kappa u(y)|}{
\mathrm{dist}_g(x,y)^\alpha} < \infty.
\]
In other words, every derivative is accompanied by a power of $\vep$ and the $\alpha$ H\"older seminorm has a factor of $\vep^\alpha$.

In fact, these semiclassical spaces are simply the ordinary H\"older spaces associated to the family of rescaled metrics
$\vep^{-2} g = g_\vep$: 
\[
\calC^{\kappa,\alpha}_\vep(M,g)  \cong \calC^{\kappa,\alpha}(M, g_\vep).
\]
Clearly each $g_\vep$ has bounded geometry of order $\ell + \alpha'$, and the bounds are uniform as $\vep \to 0$. 
\begin{remark}
From these last remarks, it is clear that in the definition of the semiclassical H\"older seminorm above, we may take
the supremum either over all balls $B$ of radius $1$ or alternately of radius $\vep$ with respect to $g$.
\label{scnormballs}
\end{remark}

\subsection{Admissible differential operators}\label{sec-admiss}

We shall prove our main sectoriality estimate for any differential operator $L$ which is strongly elliptic, 
and satisfies an additional uniformity condition.

We begin by recalling that the principal symbol of $L$ of order $m$ is a smooth function $\sigma_m(L)(x, \xi)$ on $T^*M$ 
which restricts to be a homogeneous polynomial of order $m$ (matrix-valued if $L$ acts between bundles) on each fiber $T_x^*M$.
There are various invariant ways to define this principal symbol, but the most familiar is that it is obtained by dropping
all terms of order less than $m$ and then replacing each derivative $\del_x^\alpha$, $|\alpha| = m$, by the monomial
$(i \xi)^\alpha$. (The factor of $i$ is customary because of the relationship of this symbol with the Fourier transform.) 

\begin{definition} We say that $L$ is \emph{strongly elliptic} if $\sigma_m(L)(x,\xi)$ has numerical range (or spectrum, if it is a matrix) 
contained in a sector in the right half-plane:
\[
\mathrm{spec}\, \sigma_m(L)(x,\xi) \subset \{\lambda \in \CC: |\arg(\lambda)| \leq \theta'  < \pi/2\}
\]
for all $(x,\xi) \in T^* M$.
\end{definition}

It is straightforward to check that if $L$ is strongly elliptic, then its order $m$ is even. 
If $L$ is symmetric and has real-valued coefficients, then $\sigma_m(L)(x,\xi)$ is real-valued.  For example, if 
\[
L = (\nabla^* \nabla)^{m/2} + \text{lower order terms},
\]
then $\sigma_m(L)(x,\xi) = |\xi|^m$ (or $|\xi|^m$ times the identity matrix), which clearly satisfies this condition. 

\begin{definition}
\label{def-admiss}
We say that $L$ is uniform of order $\ell + \alpha'$ (relative to a metric $g$ with bounded geometry of order $\ell + \alpha'$ on 
a manifold $M$) if the following two conditions are satisfied:
\begin{itemize}
\item[i)] the pullback by $\exp_q$ of the operator $L$ has coefficients bounded in $\calC^{\ell, \alpha'}$ in each ball 
$B_{r_0}(q)$; 
\item[ii)] there exists a closed cone $\Gamma$ strictly contained in $\{\lambda \in \CC: \mathrm{Re}\, \lambda > 0\} \cup \{0\}$ such that
if $\zeta \in \CC \setminus \Gamma$, then the endomorphism $\zeta I - \sigma_m(L)(x,\xi)$ is invertible, with the inverse satisfying
\[
|| (\zeta I - \sigma_m(L)(x,\xi))^{-1}|| \leq C (1 + |\xi|)^{-m}
\]
for some fixed $C$ (depending on $\zeta$) for all $(x,\xi)$ in the  cotangent bundle of $M$.
\end{itemize}
\label{admop}
\end{definition}
\begin{definition}
The elliptic differential operator $L$ is called admissible if it is both strongly elliptic and uniform of order $\ell + \alpha'$.
\label{admop2}
\end{definition}

It is important for our purposes that admissibility is preserved under passage to a limiting space:
\begin{lemma}
Suppose that $(M,g)$ has bounded geometry of order $\ell + \alpha'$, and let $p_j$ be a diverging sequence 
of points in $M$. Let $L$ be an admissible differential operator on $M$, as described above.  If $(M, g, p_j)$ converges 
to some limiting space $(M_\infty, g_\infty, p_\infty)$ in $\calC^{\ell, \alpha''}$, then (some subsequence of) the 
restrictions of the operator $L$ to balls $B_R(p_j)$ converges in $\calC^{\ell, \alpha''}$, as both $j \to \infty$ and $R \to \infty$ 
to an operator $L_\infty$ on this limiting space, and any such limiting operator $L_\infty$ obtained in this way is uniform 
on its space of definition.
\label{adm-persists}
\end{lemma}
\begin{proof}
By a diagonalization process using Arzela-Ascoli, it is clear that $L$ induces a limiting operator $L_\infty$ on $M_\infty$, and 
that $L_\infty$ is again strongly elliptic. (Its coefficients are only in $\calC^{\ell, \alpha''}$ for any $\alpha'' < \alpha'$.) 
\end{proof}

There is an additional important relationship between $L$ and its limiting operators:
\begin{prop}
There is an inclusion
\[
\bigcap \res_{\calC^{k,\alpha}}\, (L_\infty) \supset \res_{\calC^{k,\alpha}}\,(L), 
\]
or equivalently, 
\[
\bigcup \spec_{\calC^{k,\alpha}} (L_\infty) \subset \spec_{\calC^{k,\alpha}}(L),
\]
In both cases, the intersection or union is over all possible limiting spaces $M_\infty$ and model operators $L_\infty$. 
\label{limspecrel}
\end{prop}
\begin{proof}
Using this second formulation, suppose that $\lambda \in \spec_{\calC^{k,\alpha}}(L_\infty)$ for some limit $L_\infty$,
i.e., $\lambda I - L_\infty$ is not boundedly invertible on $\calC^{k,\alpha}(M_\infty, g_\infty)$.  In the following considerations, 
note that the natural domain of this unbounded map is $\calC^{k+m,\alpha}(M_\infty, g_\infty)$. 
There are three ways that invertibility might fail: 
\begin{itemize}
\item[a)] there exists a nontrivial function $u \in \calC^{m+k,\alpha}$ such that $L_\infty u = \lambda u$;
\item[b)] the range of $\lambda I - L_\infty$ is dense in $\calC^{k,\alpha}$, but not closed;
\item[c)] the closure of the range of $\lambda I - L_\infty$ is equal to some proper closed subspace of $\calC^{k,\alpha}$. 
\end{itemize}
We show that each of these three possibilities is incompatible with the assumption that $\lambda \not\in \spec_{\calC^{k,\alpha}}(L)$. 
In the first case, suppose that $u$ is a $\calC^{m+k,\alpha}$ solution of this limiting equation. 
Choose a sequence of radii $R_i \to \infty$, and let $\chi_i$ be a sequence of smooth cutoff functions on $M_\infty$
such that 
\[
\chi_i = \begin{cases} & 1\ \mbox{on}\ B_{R_i/2}(p_\infty) \\ &  0\ \mbox{outside}\ B_{R_i}(p_\infty) \end{cases},
\]
$0 \leq \chi_i \leq 1$ everywhere, and
\[
|\nabla^q \chi_i| \leq C/R_i^q
\]
for $q \leq m$ and with $C$ independent of $i$. For fixed $i$, the ball $B_{R_i}(p_\infty)$ in $M_\infty$ is a limit as $j \to \infty$ of balls $B_{R_i}(p_j)$ 
in $M$. Using this, we may transplant $\chi_i u$ to $B_{R_i}(p_j)$ and then extend it to equal $0$ on the rest of $M$. 

We now compute that
\[
(\lambda I - L) (\chi_i u) := h_i =\chi_i ( \lambda I - L_\infty) u + \chi_i (L_\infty - L) u + [L, \chi_i] u.
\]
Clearly, there exist constants $c_1, c_2$ such that $0 < c_1 \leq ||\chi_i u||_{k,\alpha}  \leq c_2$, uniformly in $i$, and using the 
limiting properties in this construction, $||h_i||_{k,\alpha} \to 0$.   Since the $\chi_i u$ span an infinite dimensional space,
we conclude that $(\lambda I - L)$ does not have closed range on $\calC^{k,\alpha}$, contrary to the choice of $\lambda$.

Next suppose that $(\lambda I - L_\infty)$ has dense but nonclosed range, so it does not have a bounded inverse.
There exists a sequence $u_i$ on $M_\infty$ with infinite dimensional span such that $||u_i||_{k,\alpha} = 1$ and 
$||(\lambda I - L_\infty) u_i||_{k,\alpha} \to 0$.
Precisely the same transplantation argument used above shows that $c_1 \leq ||\chi_i u_i||_{k,\alpha} \leq c_2$ and
$||( \lambda I - L) (\chi_i u_i)||_{k,\alpha} \to 0$ on $M$, which is once again a contradiction.

Finally, suppose that the range of $(\lambda I - L_\infty)$ is equal to or at least dense in some proper closed
subspace. Then the dual operator $(\bar{\lambda} I - L_\infty^*)$ has nontrivial nullspace in $(\calC^{k,\alpha})^*$.
This dual space is distributional, of course, but since $L_\infty^*$ is elliptic, any element $v$ of its nullspace
is again as regular as the coefficients of the operator and the metric allow, and this element must be bounded
as well (else it would be easy to find some $\phi \in \calC^{k,\alpha}$ such that $\langle v, \phi \rangle$ is
undefined).  We are then in the situation of the first case, once we observe that the dual operator $L^*$ 
is admissible.

This completes the proof.
\end{proof}

\subsection{Examples of admissible operators}
\label{subsec:geomlap}
The main examples of admissible operators that we have in mind are generalized Laplacians on Riemannian manifolds with 
bounded geometry, or more generally, operators of the form $(\nabla^* \nabla)^{m/2} + S$, where $S$ is an operator
of order $m-1$ usually closely associated with the metric $g$.  We begin with the second order case.

By definition a generalized Laplacian is an operator of the form
\[
\nabla^* \nabla + \mathcal K,
\]
acting on sections of some tensor bundle $E$ over $M$ (or slightly more generally, a twisted spin bundle -- the key feature 
is that its connection is induced from the Levi-Civita connection for $g$.) Here $\nabla^*$ is the adjoint of the covariant 
derivative with respect to the natural inner product on each 
fiber of $E$ and with respect to the the volume form $dV_g$.  The term $\mathcal K$ is a symmetric endomorphism 
of $E$ obtained via contractions of sums of tensor products of the curvature tensor and its covariant derivatives.

The following is a list standard examples of such operators: 
\begin{itemize}
\item[a)] The scalar Laplacian $\Delta_g = \nabla^* \nabla$ acts on the trivial rank $1$ bundle; slightly
more generally, we also consider the Hodge Laplace operator $\Delta_{g,k} = d \delta + \delta d$ acting on sections
of the bundle of exterior $p$-forms, $p = 0, \ldots, n$. The original Weitzenb\"ock formula states that
\[
\Delta_{g,p} = \nabla^* \nabla + \mathcal K_p,
\]
where $\mathcal K_p$ is an endomorphism of $\bigwedge^pM$ constructed from the curvature tensor. For example, 
$\mathcal K_1 = \mathrm{Ric}$, considered as a symmetric endomorphism on $1$-forms.
\item[b)] Next, consider the Lichnerowicz Laplacian
\[
\nabla^* \nabla  + 2 (\mathrm{Ric} - \mathrm{Riem})
\]
where $\mathrm{Ric}$ is the Ricci tensor and $\mathrm{Riem}$ the full curvature tensor. These act as symmetric endomorphisms 
on symmetric $2$-tensors via
\begin{equation*}
\begin{split}
h_{ij} & \mapsto (\mathrm{Ric}(h))_{ij} = \frac12 \left( \mathrm{Ric}_{ik} h^k_j + \mathrm{Ric}_{jk} h^k_i\right), \\
h_{ij} & \mapsto (\mathrm{Riem}(h))_{ij} = R_{ipjq}h^{pq}.
\end{split}
\end{equation*}

\item[c)] Generalizing a) in a different way is the conformal Laplacian
\[
\nabla^* \nabla + \frac{n-2}{4(n-1)} R_g,
\]
acting on scalar functions, where $R_g$ is the scalar curvature of the metric. 
\end{itemize}

\medskip

Each of the operators in the list above is symmetric,  i.e., 
\[
\langle L u, v \rangle = \langle u, L v \rangle\ \ \mbox{for}\ \ \ u, v \in \mathcal C^\infty_0(M).
\]
A classical theorem due to Chernoff \cite{Chernoff} states that because $g$ is complete, each of these has a unique self-adjoint extension
as an unbounded operator on $L^2(M, dV_g)$.  Self-adjointness guarantees that the $L^2$ spectrum lies on the real line.
If $g$ has bounded geometry of high enough order, the pointwise norm of the endomorphism $\mathcal K$ is uniformly bounded. 
Since $\nabla^* \nabla \geq 0$, we deduce that 
\[
\nabla^* \nabla + \mathcal K  \geq -C  \Longrightarrow  \spec_{L^2}( \nabla^* \nabla + \mathcal K) \subset
[-C, \infty). 
\]

There are many interesting higher order elliptic operators associated to the metric $g$.  The most well-known are the higher
order ``GJMS operators'', which generalize the conformal Laplacian above. The GJMS operator $P_m$ of order $m$ is a conformally covariant
operator which is simply equal to $(\nabla^*\nabla)^{m/2}$ if $g$ is the flat Euclidean metric, but in general has
a (complicated) set of lower order terms involving the curvature tensor.  This operator exists only for $m \leq n$ if $n$ is
even, and for all $m$ if $n$ is odd.   A (very thoroughly studied) example is the Paneitz operator
\[
P_4 = \Delta^2 - \delta ( (n-2) J -4V) d + (n-4)Q,
\]
where $V$ is the so-called Schouten tensor of the metric $g$, $J$ its trace, and $Q$ an associated scalar quantity called the $Q$-curvature. 
It is known \cite{Graham-Zworski} that these operators are symmetric, and the same reasoning as above implies that the $L^2$ spectrum
lies in a half-line $[-C,\infty)$. 

Notice that we have said nothing about the spectrum of any of these operators on $\calC^{0,\alpha}(M,g)$, and in general the relationship
between the $\calC^{0,\alpha}$ and the $L^2$ spectrum may be quite difficult to understand.   In many geometric problems, however,
we actually wish to study the action of $L$ not on $\calC^{0,\alpha}$ itself, but between some {\it weighted} H\"older spaces:
\begin{equation}
L: \mathfrak w\, \calC^{m,\alpha}(M,g) \longrightarrow \mathfrak w\, \calC^{0,\alpha}(M,g),
\label{mapwhs}
\end{equation}
where, by definition, $\mathfrak w$ is a strictly positive $\mathcal C^\infty$ (or $\calC^{\ell,\alpha'}$) function and
\[
\mathfrak w\, \calC^{\kappa,\alpha}(M,g) = \{ u = \mathfrak w\, v: v \in \calC^{\kappa,\alpha}(M,g)\}.
\]
Observe that the mapping \eqref{mapwhs} is equivalent to
\begin{equation}
L_{\mathfrak w} = {\mathfrak w}^{-1} L \mathfrak w: \calC^{m,\alpha}(M,g) \longrightarrow \calC^{0,\alpha}(M,g).
\label{mapwhs2}
\end{equation}
Observe also that since $\lambda I - L_{\mathfrak w} = \mathfrak w^{-1}( \lambda I - L) \mathfrak w$,
the spectra of \eqref{mapwhs} and \eqref{mapwhs2} are the same. 

In most of the specific examples of manifolds of bounded geometry we have given, it is possible to prove that there exist
weight functions $\mathfrak w$ for which \eqref{mapwhs} (or equivalently \eqref{mapwhs2}) is Fredholm, i.e. has
closed range and finite dimensional kernel and cokernel. It is
sometimes possible to choose $\mathfrak w$ so that $\mathfrak w\, \calC^{0,\alpha}(M,g) \subset L^2$, and if this
is the case, and if $f \in \mathfrak w\, \calC^{0,\alpha}(M,g)$, then so long as $\lambda \not\in \spec_{L^2}(L)$, there
exists a function $u \in L^2$ with $(\lambda I - L)u = f$. Local elliptic regularity implies that $u \in \calC^{m,\alpha}$ on
each ball $B_{r_0}(p)$. If we could then somehow show that $u \in \mathfrak w\, \calC^{m,\alpha}$, we would have
obtained information about $\spec_{\calC^{0,\alpha}}(L)$.  

Before doing so, we list some necessary assumptions on these weight functions.
\begin{definition}
A weight function $\mathfrak w$ is called uniform with respect to $L$ if:
\begin{itemize}
\item[a)] the mappings \eqref{mapwhs} and \eqref{mapwhs2} are Fredholm;
\item[b)] the conjugated operator $L_{\mathfrak w}$ is admissible with respect to the metric $g$.
\end{itemize}
\label{wthyp}
\end{definition}
It is clear that $L_{\mathfrak w}$ is strongly elliptic if $L$ is, since they have the same leading order terms. 
The uniformity of $L_{\mathfrak w}$ imposes a strong condition on the weight function and its derivatives.  
For example, if $L$ is the scalar Laplacian, then
\[
\Delta_{\mathfrak w} = \Delta + \mathfrak w^{-1} [\Delta, \mathfrak w] = \Delta + 
2 \frac{\nabla \mathfrak w}{\mathfrak w} \cdot \nabla + \frac{\Delta \mathfrak w}{\mathfrak w} 
\]
involves first and second derivatives of $\mathfrak w$.  Thus all these lower order terms must be uniform
in the sense we have described earlier.

\subsection{The resolvent on asymptotically hyperbolic spaces}
We conclude this section with a description of one particular setting, namely the class of asymptotically hyperbolic spaces,
where there is a somewhat more direct path to understanding the $\calC^{0,\alpha}$ spectrum of generalized Laplacians.

Suppose that $(M,g)$ is asymptotically hyperbolic, as described in Section \ref{sec:geo-bkgd}, and let $L = \nabla^* \nabla + \mathcal K$. 
Choose coordinates $(x,y)$ on $M$ near the boundary of $M$, where $x$ is a boundary defining function and $y$ is a local 
coordinate on the boundary extended to this collar neighborhood. Using this we identify the collar neighborhood with
$[0,1) \times \del M$. It is known that we can choose these functions in such a way that
\[
g = \frac{dx^2 + h(x,y)}{x^2},
\]
where $x \mapsto h(x, \cdot)$ is a family of tensors on $\del M$ in this collar neighborhood decomposition, and $h(0,y)$
is any prescribed metric representing the conformal class on $\del M$ associated to $g$. 

We have already noted that the $L^2$ spectrum of $L$ lies in some half-line $[-C,\infty)$, hence $\res_{L^2}(L) \supset \CC \setminus [-C,\infty)$.
Let $R_L$ denote both the resolvent of $L$ as an abstract operator, but also its Schwartz kernel, which is a distribution on $M \times M$. This
distribution, $R_L(\lambda; x, y, \tilde{x}, \tilde{y})$, is singular along the diagonal, where $x = \tilde{x}$ and $y = \tilde{y}$, and has additional
singularities along the boundaries where $x \to 0$ or $\tilde{x} \to 0$. The nature of these singularities can be understood in a very
detailed way using the methods of geometric microlocal analysis. We refer to \cite{MazzeoEdge} for the construction of this distribution.

The question we wish to consider is whether there exists a range of $\mu$ for which
\begin{equation}
R_L:  x^\mu \calC^{k,\alpha}(M,g) \longrightarrow x^\mu \calC^{k,\alpha}(M,g),
\label{resahmu}
\end{equation}
is bounded if $\mathrm{Re}\,\lambda \leq -C$, or equivalently, whether the conjugated Schwartz kernel 
\begin{equation}
x^{-\mu} R_{L}(\lambda; x,y, \tilde{x}, \tilde{y}) (\tilde{x})^\mu
\label{conjker}
\end{equation}
acting by convolution induces a bounded mapping on $\calC^{0,\alpha}$. 

This is true for $\mu$ lying in a certain interval determined by the so-called indicial roots of $L$. To describe this more carefully,
recall that an indicial root of $L$ is a number $\gamma \in \CC$ such that 
\[
L (x^\gamma u(x,y)) = \mathcal O(x^{\gamma+1}),
\]
where $u$ is any function smooth up to $\del M$. (We consider only the scalar case for notational simplicity.)  It is not hard
to see that this can only happen if there is some leading order cancellation, which depends only on some algebraic condition
determined by $\gamma$ and the values at $x=0$ of certain of the coefficients of $L$.  For a second order scalar operator,
this algebraic condition is a quadratic polynomial in $\gamma$, and hence there are two indicial roots. For a higher order
operator or system, there are more.  Again in this second order setting where $L$ is assumed to be symmetric on $L^2$, 
these indicial roots take the form
\[
\gamma^\pm =  \frac{n-1}{2} \pm \zeta_0
\]
for some $\zeta_0$ which is either real and nonnegative or else purely imaginary. We can define the indicial roots of $\lambda I - L$ in 
the same way, and write these as
\[
\gamma^\pm(\lambda) = \frac{n-1}{2} \pm \zeta_0(\lambda).
\]
If $\lambda$ is real and sufficiently negative, then $\zeta_0(\lambda) > 0$. There exists some $C_0 \in \RR$ such that if
$\lambda > C_0$ then $\zeta_0(\lambda)$ is purely imaginary. If $\lambda \in \CC \setminus [C_0, \infty)$, then 
$\mathrm{Re}\, \zeta_0(\lambda) > 0$ and tends to infinity as the distance from $\lambda$ to $[C_0,\infty)$ gets larger. 
One consequence of this is that any $\lambda > C_0$ lies in the continuous $L^2$-spectrum of $L$, see \cite{MazzeoUniCtn}.

Fix the half-plane $\mathrm{Re}\, \lambda \leq -B$, and define 
\[
\delta = \inf_{\mathrm{Re}\, \lambda \leq -B}  \mathrm{Re} \, \zeta_0(\lambda). 
\]
Note then, by the remarks immediately above, we can increase $\delta$ by increasing $B$ along the real axis.

We need a key structural theorem from \cite{MazzeoEdge} about the pointwise behavior of the Schwartz kernel of $R_L(\lambda)$: 
\begin{prop} \label{prop:AHresolventdecay}
The distribution $R_L(\lambda; x, y, \tilde{x}, \tilde{y})$ satisfies
\begin{equation*}
\begin{split}
& | R_L( \lambda; x,y, \tilde{x}, \tilde{y})| \leq C x^{\frac{n-1}{2} + \delta},\ \ x \to 0,\ \tilde{x} \geq c > 0; \\
& | R_L( \lambda; x,y, \tilde{x}, \tilde{y})| \leq C \tilde{x}^{\frac{n-1}{2} + \delta},\ \ \tilde{x} \to 0,\ x \geq c > 0; 
\end{split}
\end{equation*}
\end{prop}
Actually, there is a considerably sharper structural theorem which also describes the precise behavior of $R_L$ as $x, \tilde{x} \to 0$, but
for the present purposes we do not need this. 

We now state and prove a basic result on the spectrum of $L$ acting on weighted H\"older spaces.  
\begin{theorem} \label{thm:ah} Let $(M^n,g)$ be an asymptotically hyperbolic space where $L$ is a generalized Laplacian 
acting on some tensor bundle over $M$.  For every $\delta > 0$, there exists $\omega = \omega(\delta) > 0$ so that if
\begin{align} \label{eqn:loc-interval}
\mu \in \left(\frac{n-1}{2} -\delta,\frac{n-1}{2} + \delta\right),
\end{align}
then $\res_{x^{\mu} \calC^{0,\alpha}}(L) \supset \{ \lambda: \Re \lambda \leq -\omega\}$.
\end{theorem}
\begin{proof}
For simplicity we confine our discussion to the scalar case.

\

Given a choice of $\delta > 0$, by the defining remarks above concerning $\zeta_0$, we need only increase the value of $\omega$ so to ensure the interval of equation \eqref{eqn:loc-interval} is non-empty.  Then, by Proposition \ref{prop:AHresolventdecay}, the conjugated kernel \eqref{conjker} decays at least at the rate $\frac{n-1}{2} + \delta + \mu$ as $\tilde{x} \to 0$ 
and at least like $\frac{n-1}{2} + \delta - \mu$ as $x \to 0$.  We must determine whether
\[
(\tilde{x}/x)^\mu R_L(\lambda; x, y, \tilde{x}, \tilde{y}):  \calC^{0,\alpha}(M,g) \longrightarrow \calC^{0,\alpha}(M,g)
\]
is bounded, where this kernel acts by
\[
u(x,y) \longmapsto \int (\tilde{x}/x)^\mu R_L(\lambda; x, y, \tilde{x}, \tilde{y}) u(\tilde{x}, \tilde{y})\, \frac{d\tilde{x}d\tilde{y}}{\tilde{x}^n}.
\]
(The singular measure is uniformly equivalent to the $L^2$ measure for $g$.) Since $u$ does not decay, it is certainly necessary that the
product of these factors (including the singular Jacobian factor) is bounded by $\tilde{x}^{-1 + \epsilon}$ for some $\epsilon > 0$.
This implies the necessity of 
\[
\frac{n-1}{2} + \delta + \mu  - n > -1 \Leftrightarrow \mu > \frac{n-1}{2} - \delta.
\]
On the other hand, the rate of growth (or decay) of the output is determined by the behavior of this conjugated kernel in two regions: 
the first as both $x, \tilde{x} \to 0$ and the second as $x \to 0$, $\tilde{x} > 0$.   We refer to \cite[Theorem 3.27]{MazzeoEdge} for
the precise explanation.  The kernel is bounded in this first regime, and is bounded by $x^{ (n-1)/2 + \delta - \mu}$ in the second. 
This is then bounded if and only if 
\[
\mu < \frac{n-1}{2} + \delta.
\]
The case of equality must be omitted here because in that special case the solution might behave like $\log x$ as $x \to 0$. 

Altogether then, we have argued that
\[
(\tilde{x}/x)^\mu R_L(\lambda; x, y, \tilde{x}, \tilde{y}):  \calC^{0,0}(M,g) \longrightarrow \calC^{0,0}(M,g)
\]
is bounded if and only if
\begin{equation}
\frac{n-1}{2} - \delta < \mu < \frac{n-1}{2} + \delta.
\label{limitsmu0}
\end{equation}
Boundedness on H\"older spaces now follows readily from elliptic estimates.
\end{proof}

We have described the asymptotically hyperbolic case in some detail, but note that it is possible to prove similar results in the asymptotically conic, cylindrical,
complex hyperbolic and symmetric or homogeneous cases described above.  There is no such argument (to our knowledge) for
infinite covers of compact manifolds. 

This material was included to indicate that the spectral hypothesis of admissibility can be verified in different ways.
In Section \ref{scpc} we describe a different sort of parametrix construction which turns out to be sufficient for our purposes.
It works for arbitrary manifolds of bounded geometry, but only if $\lambda$ has sufficiently large negative real part.
In many ways, that parametrix construction is simpler than the one needed in the AH case, but is perhaps less familiar. 

\section{Sectoriality of admissible operators on H\"older spaces}
\label{sec:proofThm}
We now turn to the proof of the main sectoriality theorem, Theorem \ref{thm:main-A}.

The first key observation is that if $L$ is an elliptic differential operator of order $m$, then the sectoriality of $L$ is equivalent to a certain estimate for
the {\it semiclassical} resolvent of $L$, which means the following. Define $\varepsilon = |\lambda|^{-1/m}$ and $\zeta = \lambda/|\lambda|$.
We then rewrite the operator that appears in  Lemma \ref{lem:halfplane} characterizing sectoriality as
\[
\lambda (\lambda I - L)^{-1} = \frac{\lambda}{|\lambda|}  ( (\lambda/|\lambda|) I - |\lambda|^{-1} L)^{-1} = \zeta ( \zeta I - \varepsilon^m L)^{-1}.
\]
Disregarding the harmless unit prefactor $\zeta$, the operator $(\zeta I - \varepsilon^m L)^{-1}$ is called the semiclassical
resolvent of $L$.  Since $\mathrm{Re}\, \lambda \leq \omega$ and we may as well assume that $\omega < 0$,
we only need consider such $\zeta$ with $\mathrm{Re}\,\zeta < 0$.  Altogether then, an equivalent formulation of the sectoriality of $L$ is 
that 
\begin{equation}
(\zeta - \varepsilon^m L)^{-1}: \calC^{k,\alpha}(M,g) \longrightarrow \calC^{k,\alpha}(M,g)
\label{scformulation}
\end{equation}
exists and has norm which is uniformly bounded independently of $\vep \in (0, \vep_0)$, for some $\vep_0 >0$, and $\zeta$ with
$|\zeta| = 1$, $\mathrm{Re}\, \zeta \leq -c < 0$.   In other words, 
\begin{quotation}
{\bf sectoriality is equivalent to the uniform boundedness of
the {\it semiclassical} resolvent on regular, i.e., {\it non-semiclassical}, function spaces.}
\end{quotation}
Part of our assertion is that this resolvent exists as a bounded operator provided $\vep$ is sufficiently small, i.e., 
$\lambda I - L$ is invertible on $\calC^{k,\alpha}(M,g)$ when $\mathrm{Re}\, \lambda$ is sufficiently negative. 

As explained at the end of this section, the case $k > 0$ follows from the case $k=0$, so we assume that $X = \calC^{0,\alpha}$
until the final part of this section.

We prove this in a series of steps, outlined here and carried out in the rest of this section. In the first step, we show that $(\zeta I - \vep^m L)^{-1}$ exists as a bounded operator for 
each sufficiently small $\vep$ and for every $\zeta$ with $\mathrm{Re}\, \zeta < 0$, but with no claim about uniformity. 
Here it is irrelevant whether standard or semiclassical H\"older spaces (as defined in Section \ref{subsec:funct-spcs}) are used since they are equivalent for each fixed 
$\vep$. This is a `perturbative result', and follows from the existence of a semiclassical parametrix. We state this result carefully in
the next proposition, and for the reader's convenience sketch this parametrix construction using geometric microlocal analytic techniques in Section \ref{scpc}; this methodology is explained there.
Next we recall the `easy' semiclassical elliptic Schauder estimate  associated to any strongly elliptic semiclassical family 
$\zeta I - \vep^m L$. This is equivalent to the uniform boundedness of the inverse of this semiclassical operator between 
{\it semiclassical} H\"older spaces. The main step of the whole proof is to upgrade this to an estimate between standard 
H\"older spaces $\calC^{0,\alpha}$ which is independent of $\vep > 0$. This is done by establishing uniform bounds for 
this semiclassical resolvent as a map $\calC^j \to \calC^j$ for $j = 0, 1$, and applying interpolation.

\begin{prop} Let $L$ be an admissible  operator of order $m$ on a manifold $(M,g)$ of  bounded geometry.  
Then the unbounded operator $\zeta I - \vep^m L: \calC^{0,\alpha}_\vep(M,g) \to \calC^{0,\alpha}_\vep(M,g)$ has a bounded inverse 
for each sufficiently small $\vep > 0$ and for each $\zeta$ with $|\zeta| = 1$ and $\mathrm{Re}\, \zeta < 0$.
\label{scparprop}
\end{prop}
\begin{proof}
This proof uses the machinery of microlocal analysis extensively. A detailed introduction to these methods is provided in Section \ref{scpc}. As is carefully defined and explained in that section, there exists a parametrix $G_\vep$ for $\zeta I - \vep^m L$ which is an element of order $-m$ in 
the semiclassical uniform pseudodifferential calculus $\Psi^{*,*}_{\mathrm{sc-unif}}(M,g)$ (see Definition \ref{def:sc-pseud}). Thus for each $\vep$, $G_\vep$ is an approximate
inverse for $\zeta I - \vep^m L$; its discrepancy from being an exact inverse is captured by the `residual operators'
$Q_{j,\vep} \in \Psi^{-\infty,\infty}_{\mathrm{sc-unif}}(M,g)$, $j = 1, 2$, via the identities
\[
(\zeta I - \vep^m L) G_\vep = I - Q_{1,\vep}, \quad G_\vep (\zeta I - \vep^m L) = I - Q_{2,\vep}.
\]
Each $Q_{j,\vep}$ is a smoothing operator (this is the meaning of the first superscript $-\infty$) with Schwartz kernel supported in 
some fixed neighborhood of the diagonal $\{(z,\tilde{z}): \mathrm{dist}_g(z, \tilde{z}) \leq C\}$, and vanishing to all orders
as $\vep \searrow 0$ (which is the meaning of the second superscript, $+\infty$). 

In more detail, $G_\vep$ and the $Q_{j,\vep}$ are one-parameter families of operators, with Schwartz kernels $G(\vep, z, \tilde{z})$ and
$Q_{j}(\vep, z, \tilde{z})$, $z, \tilde{z} \in M$. For each $\vep > 0$, $G(\vep, z, \tilde{z})$ is an ordinary pseudodifferential operator 
of order $-m$ which is a parametrix for $\zeta I - \vep^m L$, and the $Q_j$ are the smoothing error terms.  These Schwartz kernels
vary smoothly in $\vep$ for $\vep > 0$.  The important new feature is their behavior as $\vep \to 0$.  First, if $z \neq \tilde{z}$, then
$G(\vep, z, \tilde{z})$ and $Q_j(\vep, z, \tilde{z})$ decay faster than $\vep^N$ for any $N$. This convergence is uniform in any region where
$\mathrm{dist}_g(z, \tilde{z}) \geq c'' > 0$ for any $c'' > 0$.   The behaviour of $G(\vep, \cdot, \cdot)$ near the diagonal as $\vep \to 0$ 
requires a bit more work to describe; this is done in Section \ref{scpc}. On the other hand, for $j = 1, 2$, $Q_{j}(\vep, z, \tilde{z}) 
\in \calC^\infty([0,\vep_0) \times M^2)$, and these kernels decay rapidly along with all derivatives as $\vep \to 0$, uniformly on $M \times M$.
This construction works assuming that $\mathrm{Re}\, \zeta < 0$, but the rate of decay (which is actually exponential) diminishes
as $\mathrm{Re}\, \zeta \to 0$.

It is straightforward to deduce from this structure that $||Q_{j,\vep}||_{\mathcal L(\calC^{0,\alpha})} \to 0$ as $\vep \to 0$.
Hence both $(I - Q_{1,\vep})^{-1}$ and $(I - Q_{2,\vep})^{-1}$ exist as bounded operators on $\calC^{0,\alpha}$ for any fixed sufficiently small
$\vep > 0$. Thus we can write
\[
(\zeta I - \vep^m L)^{-1} = G_\vep (I - Q_{1,\vep})^{-1} = (I - Q_{2,\vep})^{-1} G_\vep ; 
\]
this proves that $(\lambda I - L)^{-1}$ exists for any $\lambda$ with $\mathrm{Re}\, \lambda$ is sufficiently negative. 
\end{proof}

In the next step we continue in this same semiclassical vein: 
\begin{prop}[Semiclassical elliptic estimate] \label{lemma:sc} There exists a constant $C > 0$ such that for all $\vep \in (0,1)$,
$u \in \scholder^{2,\alpha}$ and $\zeta$ with $|\zeta| = 1$ and $\mathrm{Re}\, \zeta \leq c' < 0$, 
\begin{equation} \|u\|_{m,\alpha,\vep} \leq C \left( \| (\zeta I - \vep^m L) u\|_{0,\alpha,\vep} + \sup_M |u| \right). 
\label{scholderest}
\end{equation}
\end{prop}
\begin{proof} 
As already hinted in the definition given earlier of semiclassical H\"older spaces, we show that \eqref{scholderest} is simply the
`standard' Schauder estimate relative to the metric $g_{\vep} = \vep^{-2} g$.  

If $z$ is a normal coordinate for $g$ in a ball $B(r_0)$ of radius $r_0$ about any point, then $w = z/\vep$ is a 
normal coordinate for the metric $g_\vep = \vep^{-2}g$ on a ball of radius $r_0/\vep$. Indeed, 
\[
g = g_{ab}(z) dz^a dz^b= (\delta_{ab} + E_{ab}(z)) dz^a dz^b, \qquad |E|_g = \calO(|z|^2),
\]
hence
\[
g_{\vep}= (\vep^{-2} \delta_{ab} + \vep^{-2} E_{ab}(\vep w)) dz^a dz^b = (\delta_{ab} + E_{ab}(\vep w)) dw^a dw^b.
\]
Using the hypothesis of bounded geometry, this shows that the coefficients and derivatives of $g_{\vep}$ are uniformly controlled 
in the rescaled normal coordinates. Denoting this dilation operator on a given ball by $S_\vep$, then the admissibility hypothesis
shows that the rescaled operator $\vep^m S_\vep^*L$ is strongly elliptic on each such ball.

The standard local elliptic estimate on any geodesic ball $B$ for $g_{\vep}$ states that if $B'$ is the ball of half the radius
and same center, then for all $u \in \calC^{m,\alpha}(g_{\vep})$, 
\[ 
\| u \|_{B', m, \alpha, g_{\vep}} \leq C \left( \|(\zeta I - \vep^m L) u\|_{B, 0, \alpha, g_{\vep}} + \sup_B |u| \right).
\]
Taking the supremum of the right hand side over all balls $B$ of fixed radius (provided by bounded geometry),
and then taking the supremum over the corresponding balls $B'$ on the left yields the global estimate.  Crucially, 
the  constant $C$ is independent of $\vep \leq 1$.

The proof is now completed by observing that the usual H\"older norm for $\calC^{m,\alpha}(g_{\vep})$ is precisely the same as
the semiclassical H\"older norm for $\calC^{m,\alpha}_\vep(g).$
\end{proof}

We now establish the $\calC^0$ version of the sectoriality estimate. 
\begin{prop}
There is a constant $C > 0$ such that
\[
||u||_{0} \leq C ||(\zeta I - \vep^m L)u||_{0}
\]
for all unit $\zeta$ with $\mathrm{Re}\, \zeta \leq c' < 0$, $\vep \in (0,1]$ and $u \in \calC^{m,\alpha}(M,g)$. 
\label{prop3.3}
\end{prop}
\begin{proof} 
If this assertion were false, there would exist sequences $\zeta_j$, $\vep_j $ and $u_j \in \calC^{m,\alpha}$ such that 
\begin{equation} 
\| u_j \|_{0} > j \| ( \zeta_j - \vep_j^m L) u_j \|_0
\end{equation}
Replace $u_j$ by $v_j = u_j / \| u_j \|_0$ and set $f_j =  ( \zeta_j - \vep_j^m L ) v_j$, so that $\|v_j\|_0=1$ 
and $\|f_j\|_0  \leq 1/j \to 0$.   Passing to a subsequence, we assume that $\zeta_j \to \zeta_*$. 

Next, choose a point $p_j \in M$ where $|v_j(p_j)| > 1/2$.  By virtue of the bounded geometry of $(M,g)$, the
restriction of the metric $g$ and the coefficients of the operator $L$ (expressed in normal coordinates) 
on the balls $B_j := B_{r_0}(p_j)$ are bounded in $\calC^{m,\alpha}$, independently of $j$. 

There are two cases to consider.  First suppose that $\vep_j \to \vep_* > 0$.  If the $p_j$ remain
in a compact set of $M$, then it is straightforward to extract a limit $v \in \calC^{0,\alpha}$ of the sequence $v_j$ which 
is not identically zero, and which satisfies $(\zeta_*  I - \vep_*^m L) v = 0$.  This is impossible given our hypothesis
that $\vep_*^{-m}\zeta_*$ does not lie in the spectrum.

Now turn to the case where the sequence $p_j$ diverges.  
Following the discussion in 'Pointed limits of manifolds of bounded geometry' in Section \ref{sec:geo-bkgd}, choose a subsequence so that $(M,g, p_j)$ converges to a limiting space 
$(M_\infty, g_\infty, p_\infty)$ as pointed Riemannian spaces in the $\calC^{\ell,\alpha'}$ topology, and so that $L$ converges
in this construction to an operator $L_\infty$ on $M_\infty$.  By Lemma~\ref{adm-persists}, $L_\infty$ is admissible.   
Since $1/2 \leq |v_j(p_j)| \leq 1 = \sup |v_j|$, 
it follows as before that some subsequence of these functions converges to a nontrivial limiting function $v_\infty$ on $M_\infty$
which satisfies $(\zeta_* I - \vep_*^m L_\infty) v_\infty = 0$.  Clearly $|v_\infty| \leq 1$ everywhere and by local 
Schauder estimates, there exists a constant $C$ such that $1 \leq ||v_\infty||_{0,\alpha} \leq C$. 

We may now employ exactly the same transplantation argument as in the proof of Lemma \ref{adm-persists} to show that
the existence of this solution $v_\infty$ contradicts the fact that $\vep_*^{-2} \zeta_*$ does not lie
in the spectrum of $L$.  The only minor modification from the proof is that we now write
\begin{multline*}
( \zeta_j I - \vep_j^m L) ( \chi_i v_{\infty} ) := h_i =  \chi_i ( \zeta_* I -\vep_*^m L_{\infty}) v_{\infty} + \\
(\zeta_j  - \zeta_*) \chi_i v_{\infty} - \chi_i (\vep_j^m - \vep_*^m) L_{\infty} v_{\infty} - 
\chi_i \vep_j^m (L - L_{\infty}) v_{\infty} -\vep_j^m [ L, \chi_i ] v_{\infty},
\end{multline*}
but this tends to $0$ in norm, as before.     Thus we reach a contradiction in this case too.

We have now reduced to the case where $\vep_j \to 0$. 
If $z$ is a normal coordinate in $B_j$, then as discussed above, $w = z/\vep_j$ is a normal coordinate for the metric 
$g_j = \vep_j^{-2}g$ on a ball of radius $r_0/\vep_j$. Indeed, 
\[
g_j= (\delta_{ab} + E_{ab}(\vep_j w)) dw^a dw^b,\ \ E_{ab}(\vep_j w) = \mathcal O( \vep_j^2C^2) \ \ \mbox{for} \ |w| \leq C.
\]
Thus $g_j$ converges uniformly to the Euclidean metric on any compact set of $\RR^n$. 

A very similar computation shows that if $L_j$ denotes the operator $L$ expressed in these rescaled coordinates, then
\[
\vep_j^m L_j \to  L_E,
\]
a constant coefficient operator on $\RR^n$; this limit is uniform on any compact subset of $\RR^n$.  
Observe that any term in $L$ involving a derivative of order less than $m$ tends to $0$ in this limit. 
In fact, using multi-index notation, 
\[
\mbox{if}\ L = \sum_{|J| \leq m} a_J(z) \del_z^J, \ \mbox{then}\ L_E = \sum_{|J| = m} a_J(0) \del_w^J.
\]
(In the language developed in \S \ref{scpc} below, $L_E$ is simply the constant coefficient operator associated to the semiclassical
symbol $\sigma_m^\scl(L)(z,\xi)$ of $L$ at $z=0$; we refer to that section for more on this.) 

Now, passing to a subsequence, we may assume that $\zeta_j \to \zeta_*$ and that $v_j \to v$ in
$\calC^m_{\mathrm{loc}}(\RR^n)$. Clearly $|v| \leq 1$; furthermore, $1/2 \leq |v_j(0)| \leq 1$, so 
the function $v$ is nontrivial. It also must satisfy 
\[
(\zeta_* - L_E) v = 0
\]
on all of $\RR^n$.   The existence of such a bounded solution is easily ruled out by Fourier analysis.  Indeed, if we regard
$v$ as a tempered distribution, then taking the Fourier transform transforms this equation to
\[
(\zeta_* - \sigma_m(L)(0,\xi)) \hat{v}(\xi) = 0.
\]
Using our definition of strong ellipticity, and the fact that $\mathrm{Re}\, \zeta_* \leq 0$, $\zeta_* \neq 0$, the factor
$(\zeta_* - \sigma_m(L)(0,\xi))$ is invertible for all $\xi \in \RR^n$, hence $\hat{v} = 0$, so $v = 0$ as asserted.
We have arrived at a final contradiction, and have thus established the $\calC^0$ bound. 
\end{proof}

The (nearly) final step in the proof of Theorem 1 is to establish the corresponding $\calC^1$ sectoriality estimate, by reducing it to the $\calC^0$ 
estimate. 
\begin{prop} \label{prop:second-resolvent-estimate}
For any fixed $c' < 0$ sufficiently small, there exists a constant $C$ such that for all unit $\zeta$ with $\mathrm{Re}\,~\zeta~ \leq~c'~<~0$
and $\vep \in (0,1]$,
\begin{equation} 
\| u \|_{1} \leq C \| (\zeta - \vep^m L) u \|_{1}
\end{equation}for all $u \in \calC^{m+1,\alpha}(M,g)$. 
\end{prop}
\begin{proof}
Begin by differentiating both sides of the equation $(\zeta - \vep^m L) u = f$, and then commute derivatives to obtain
\[
(\zeta - \vep^m L) (\nabla u) = \nabla f - \vep^m [\nabla, L] u.
\]
The $\calC^0$ bound in the previous proposition yields that
\[
||\nabla u||_0 \leq C || \nabla f||_0 + C \vep^m || [\nabla,L]u||_0.
\]
However, $[\nabla,L]$ is a differential operator of order $m$ with uniformly bounded coefficients, and the semiclassical
estimate in Lemma \ref{lemma:sc} shows that 
\[
\vep^m || [\nabla,L] u ||_0 \leq C ||u||_{m,\alpha, \vep} \leq C' ( || f||_{0,\alpha, \vep} + \sup_M |u| ) \leq C' ||f||_1.
\]
We have used the $\calC^0$ estimate from the previous Proposition in this last inequality to 
estimate $\sup_M |u| \leq \sup_M |f|$. 

This completes the proof of Proposition \ref{prop:second-resolvent-estimate}.
\end{proof}

The proof of the sectoriality estimate on $\calC^{0,\alpha}$ is attained by applying an interpolation argument.  The two results above show that 
there exist constants $C_0, C_1$ such that
\[
||(\zeta - \vep^m L)^{-1} f||_\ell \leq C_\ell ||f||_{\ell}, \quad \ell = 0, 1,
\]
for all $f \in \calC^\infty(M,g)$ and  for all unit $\zeta$ with $\mathrm{Re}\, \zeta \leq c' < 0$, $\vep \in (0,1]$.  The space $\calC^{0,\alpha}(M,g)$ 
is identified with the interpolation space $[\calC^0(M,g), \calC^1(M,g)]_\alpha$, and from this we conclude that
\[
||(\zeta - \vep^mL)^{-1} f||_{0,\alpha} \leq C_0^{1-\alpha} C_1^{\alpha} ||f||_{0,\alpha}
\]
for all unit $\zeta$ with $\mathrm{Re}\, \zeta \leq c' < 0$, $\vep \in (0,1]$  and $u \in \calC^{\infty}(M,g)$, and hence 
in little H\"older spaces by density of $\calC^\infty$ in the little H\"older spaces.
This is the sectoriality estimate on $\calC^{0,\alpha}$.

We conclude this section by showing how sectoriality on $\calC^{k,\alpha}$, $k \geq 1$, follows from the case $k = 0$.
\begin{corollary}
If $(M,g)$ has bounded geometry of order $k + m + \alpha'$ for some $0 < \alpha' < 1$, and $L$ is an admissible operator
with coefficients uniform of order $k + \alpha'$, then $L$ is sectorial on $\calC^{k,\alpha}(M,g)$.
\label{higherregsect}
\end{corollary}
\begin{proof}
We have proved the uniform boundedness of the norm of $\lambda (\lambda I - L)^{-1}: \calC^{0,\alpha} \to \calC^{0,\alpha}$.
Suppose that $f \in \calC^{k,\alpha}(M,g)$, and write $u = u_\lambda = \lambda (\lambda I - L)^{-1} f$.  The $\calC^{0,\alpha}$ sectoriality
estimate we have just proved shows that 
$||u_\lambda||_{0,\alpha} \leq C \|f \|_{0,\alpha}$ uniformly in $\lambda$.  Furthermore, by local elliptic regularity, $u_\lambda \in \calC^{k+m,\alpha}_{\mathrm{loc}}$,
and if $B$ is any ball of radius $\frac12 r_0$ and $B'$ the ball of radius $r_0$ with the same center, then
\[
||u_\lambda||_{B, k+m,\alpha} \leq C( ||f||_{B', k, \alpha} + ||u_\lambda||_{B', 0}) \leq C ( ||f||_{B', k, \alpha} + ||u_\lambda||_{B', 0, \alpha}).
\]
Taking the supremum over all such balls, we obtain
\[
||u_\lambda||_{k+m,\alpha} \leq C (||f||_{k,\alpha} + ||u_\lambda||_{0,\alpha}) \leq C (||f||_{k,\alpha} + ||f||_{0,\alpha}) = C' ||f||_{k,\alpha},
\]
which is the sectoriality estimate on $\calC^{k,\alpha}(M,g)$.
\end{proof}

This concludes the proof of Theorem \ref{thm:main-A}.

\section{The semiclassical parametrix construction}
\label{scpc}

We now provide details about the construction of the semiclassical parametrix for the family of operators $(\zeta I - \vep^m L)$.  There are two 
novel features in our presentation. The first is a minor one: the semiclassical calculus is best documented in the setting of operators on
$\RR^n$, cf.\ \cite{Dim-Sj, Martinez}, or in more modern expositions, on certain other special manifolds \cite{Zworski}. 
It adapts easily to the setting of manifolds of bounded geometry.  The second is that, unlike the approaches in these citations,
we carry out this construction using geometric microlocal analysis. This particular application of that theory was observed by Melrose and 
partially developed in his lecture notes \cite{Mel-berkeley}.  We review this method to keep this paper relatively self-contained. 

Strictly speaking, the techniques here tacitly assume that both the operator $L$ and the Riemannian manifold $(M,g)$ are $\calC^\infty$.
Therefore we shall assume that this is the case, i.e., that all data are smooth, until near the end of this section. Only there will
we show how to obtain the main conclusion of this section, i.e., the existence of the semiclassical resolvent, when $L$ and $g$
only have finite regularity. 

\smallskip

\noindent{\bf The semiclassical double space.}  A family of operators $A = A_\vep$ is called a \emph{semiclassical family of pseudodifferential 
operators} if each $A_\vep$ is a pseudodifferential
operator in some standard calculus of such operators on $M$ (see below). The Schwartz kernel $K_{A_\vep}$ of each $A_\vep$ is a distribution 
on $M \times M$ which has a classical conormal, or polyhomogeneous, singularity along the diagonal $\mathrm{diag} \subset M^2$ (see Definition \ref{def:conormal} for the precise definitions of these terms).  
As $\vep \searrow 0$, the distribution $K_{A_\vep}$, and in particular its singularity along the diagonal, must degenerate somehow. 
The ``geometric'' microlocal way to describe this involves regarding this family of Schwartz kernels as a single 
distribution $K_A(\vep, z, \tilde{z})$ on the augmented double-space $[0,\vep_0)_\vep \times M^2$.  As such, it has a singularity
along the family of diagonals $(0,\vep_0)\times \mathrm{diag}$, and as we now describe, an extra singularity along
$\{0\} \times \mathrm{diag}$. Our goal is to describe this extra singularity. 

To do this, we pass to a slightly larger space obtained by taking the real blow up of the submanifold $\{0\} \times \mathrm{diag}$ 
in $[0,\vep_0)\times M^2$. This process -- which should be carefully distinguished not only from the process of `complex blow-up'
which is common in algebraic geometry, but also from the use of the phrase blow-up associated to various sorts of rescaling arguments
in PDE -- amounts to introducing a singular cylindrical coordinate system around this submanifold and `adding' the points where `$r=0$'
as a new boundary hypersurface.  Assuming that $z, \tilde{z}$ are identical coordinates on the two copies of $M$, introduce polar coordinates 
$r = | (\vep, z-\tilde{z})| \geq 0 $, $\omega = (\vep, z-\tilde{z})/r \in S^n_+$, and declare this new {\it semiclassical double space}, which we write
as $M^2_\scl$, to be the manifold with corners on which $(r, \omega, \tilde{z})$ is a nonsingular smooth coordinate system. 
Note that $M^2_{\scl} \setminus \{r=0\}$ is canonically isomorphic to $[0,\vep_0) \times M^2 \setminus (\{0\} \times \mathrm{diag})$, 
but each point on this diagonal at $\vep=0$ is replaced by its inward-pointing spherical normal bundle at that point. The entire hypersurface 
$\{r=0\}$ is called the front face and denoted by $\ff$.  It is the total space of a fibration over $\{0\} \times \mathrm{diag}$, 
with each fiber a closed hemisphere $S^n_+$.  The points of $\ff$ correspond to directions of approach to $\{0\}\times \mathrm{diag}$. 

Although the polar coordinates give a nonsingular coordinate system, it is usually more convenient to use projective coordinates instead. 
Therefore we introduce 
\[
\vep, w = (z-\tilde{z})/\vep, \tilde{z}
\]
as a coordinate system on $M^2_\scl$ away from the `original' face $\{\vep = 0\} \times (M^2 \setminus \mathrm{diag})$. These are
defined and smooth up to points in the interior of $\ff$, but are undefined at $\vep=0$ away from the diagonal.  In their region of 
definition, $\vep$ serves as a defining function for $\ff$.  Using these, the interior of each hemisphere fiber in $\ff$ is identified 
with $\RR^n$, with linear coordinate $w$. In fact, this projective identification of each fiber of $\ff$ with $\RR^n$ is {\it well-defined 
up to linear transformations}.  In other words, if $z'$ is any other choice of local coordinates, and $\tilde{z}'$ the same coordinate 
system on the second copy of $M$, then $w' = (z' - \tilde{z}')/\vep = A w + \calO(\vep)$, for some $A \in \mathrm{Gl}_n$,
and hence $w' = Aw$ at $\vep = 0$, where $A$ is a matrix (that may depend smoothly on $\tilde{z}$, i.e., vary with the hemisphere fibers). 

Let $\pi: M^2_{\scl} \to [0,\vep_0)$ be the composition of blowdown $M^2_{\scl} \to [0,\vep_0)\times M^2$ followed by projection to
the first factor.  Each level set $\pi^{-1}(\vep)$, $\vep > 0$, is a copy of $M^2$. The preimage $\pi^{-1}(0)$ is the union of two 
manifolds with boundary: namely, the manifold with boundary obtained by blowing up $M^2$ along its diagonal, and the 
closure of $\ff$, which is a bundle of closed hemispheres. The intersection of these two hypersurfaces is naturally identified 
with the spherical normal bundle of the diagonal in $M^2$.   

\medskip
\noindent{\bf Lifts of semiclassical differential operators}
The blowup construction in this particular setting is motivated by a simple computation:  consider the lift of $\zeta I - \vep^m L$ first to 
the left factor of $M$ in $M^2$ (i.e., differentiating in $z$ rather than $\tilde{z}$), then to $[0,\vep_0)\times M^2$ and finally to $M^2_\scl$.  
To compute this lift, first observe that 
\[
\vep \del_{z_j} = \del_{w_j};
\]
hence, writing $L = \sum_{|\beta|} a_\beta(z) \del_z^\beta$ using that $z = \tilde{z} + \vep w$, we arrive at
\[
\vep^m \sum_{|\beta| \leq m} a_\beta(z) \del_z^\beta  = \sum_{|\beta| = m} a_\beta(\tilde{z}) \del_w^\beta + 
\sum_{|\beta| \leq  m} c_\beta(\tilde{z},w, \vep) \vep^{m-|\beta|}\del_w^\beta.
\]
The coefficients $c_\beta(\tilde z, w, \vep)$ depend smoothly on all three variables, and arise from the Taylor expansions in $\vep$ of 
the functions  $a_\beta(\tilde{z} + \vep w)$. The first set of terms on the right in this equality are the terms of order $0$ in
this Taylor expansion when $|\beta| = m$; these contain no positive powers of $\vep$. On the other hand, the coefficient functions 
with $|\beta| = m$ in the second sum on the right arise only from the higher order terms of the Taylor expansions of 
$a_\beta(\tilde{z} + \vep w)$ for the same $\beta$, and hence these vanish at least like $\vep$.  The key point in this calculation is that the
the lift of $\vep^m L$ is the sum of a homogeneous operator of order $m$ with coefficients which depend only on $\tilde{z}$, and hence
has {\it constant coefficients} on each fiber of $\ff$ at $\vep = 0$, and a remainder term which is a differential operator of order $m$ with 
coefficients depending smoothly on $\vep$, $w$ and $\tilde{z}$, and which vanishes at $\vep=0$.  Using this, we now write 
\[
\zeta I - \vep^m L =  (\zeta I - \sigma_m^{\scl}(L)(\del_w)) + \text{(error term vanishing at $\vep=0$)}.
\]
Here, by definition, the {\it semiclassical symbol} of $L$, $\sigma_m^{\scl}(L)(\del_w)$, is the constant coefficient operator 
which appears in the computation above.  A standard microlocal notation is to set $D_{w_j} = (1/i) \del_{w_j}$; under Fourier
transform this corresponds to the linear variable $\xi_j$.  This semiclassical symbol is closely related to the principal symbol
of $L$ by
\[
\sigma_m \left(\sigma_m^{\scl}(L)( (1/i)\del_w) \right) = \sigma_m(L)(\tilde{z}, \xi).
\]
(Since $\sigma_m^{\scl}$ is homogeneous and $m$ is even, the factor of $i$ reduces to an even more harmless factor of $\pm 1$.)

The semiclassical symbol is well defined up to a linear change of coordinates in $w$.  It also depends smoothly on $\tilde{z}$,
but the dependence is only parametric, and we frequently drop it from the notation below. 

The key point in all that follows below is that there exists a distribution $H_\zeta(w) = \vep^{-n}\overline{H}_\zeta(w)$ (which depends smoothly 
on $\tilde{z}$) such that  
\[
(\zeta I - \sigma_m^{\scl}(L)(\del_w)) \overline{H}_\zeta(w) = \overline{H}_\zeta(w) (\zeta I - \sigma_m^{\scl}(L)(\del_w)) =\delta(w).
\]
As we now explain, this is a consequence of the strong ellipticity of $L$.  We explain this and develop some further properties of 
$H_\zeta$, in the next subsection.  The appearance of the slightly odd-looking factor of $\vep^{-n}$ will be explained there too.
The full parametrix construction itself is simply a perturbative construction which uses this distribution as the leading term of
a formal series in $\vep$ which is constructed to be a `formal' inverse of $\zeta I - \vep^m L$.  This is then readily converted
to a good parametrix with rapidly decaying error terms, and then to an actual inverse if $\vep$ is small. 

\medskip
\noindent{\bf Green function for the model problem.}  Here we establish the existence and certain properties of $\olH_\zeta(w)$.

\medskip

\begin{lemma} \label{lem:greenfunc}There exists a distribution $\olH_\zeta(w)$ on $\RR^n$ such that $\olH_\zeta(w - \tilde{w})$ is the Schwartz kernel of
a translation-invariant pseudodifferential operator on $\RR^n$ of order $-m$ and which has the the following two properties.  First
\begin{equation}
(\zeta I - \sigma_m^{\scl}(L)(\del_w)) \circ \olH_\zeta = \olH_\zeta \circ (\zeta I - \sigma_m^{\scl}(L)(\del_w) = \delta(w).
\label{modeldelta}
\end{equation}
In addition, $\olH_\zeta$ depends smoothly on $\tilde{z}$, and satisfies $|\olH_\zeta(w)| \leq C e^{-\delta |w|}$ where $\delta > 0$ provided
$\mathrm{Re}\, \zeta < 0$.
\end{lemma}
\begin{proof}  Strong ellipticity of $L$ implies the function $\xi \mapsto (\zeta I - \sigma_m^{\scl}(L)(\xi))^{-1}$ is $\calC^\infty$ and polynomially bounded,
hence is an element of $\mathcal S'$, the space of tempered distributions. As such, we may take its inverse Fourier transform and define
\[
\olH_\zeta(w) = \frac{1}{(2\pi)^n} \int_{\RR^n}  e^{iw \cdot \xi} (\zeta I - \sigma_m^{\scl}(L)(\xi))^{-1}\, d\xi.
\]
This too is an element of $\mathcal S'$.  This integral converges if $m > n$, but if this is not the case, we can make sense of it 
as the distributional limit
\[
\lim_{\delta \to 0} \frac{1}{(2\pi)^n} \int_{\RR^n}  e^{iw \cdot \xi} \chi(\delta \xi) (\zeta I - \sigma_m^{\scl}(L)(\xi))^{-1}\, d\xi,
\]
where $\chi \in \calC^\infty_0(\RR^n)$ equals $1$ in some neighborhood of the origin. In other words, 
pairing this expression with a Schwartz function $\phi(w)$, we note that the limit in $\delta$ can be taken in
the classical sense on the right side of 
\begin{multline*}
\int_{\RR^n} \left(\int_{\RR^n}  e^{iw \cdot \xi} \chi(\delta \xi) (\zeta I - \sigma_m^{\scl}(L)(\xi))^{-1}\, d\xi\right)\, \phi(w)\, dw =  
\int_{\RR^n}  \hat{\phi}(-\xi) \chi(\delta \xi) (\zeta I - \sigma_m^{\scl}(L)(\xi))^{-1}\, d\xi.\end{multline*}

Using similar standard distributional manipulations, we can also justify that
\[
(\zeta I - \sigma_m^{\scl}(L)(\del_w)) \olH_\zeta(w) = \olH_\zeta(w) (\zeta I - \sigma_m^{\scl}(L)(\del_w)) = (2\pi)^{-n} \int_{\RR^n} e^{iw\cdot \xi} \, d\xi = I. 
\]

We recall also that since $(\zeta I - \sigma_m^{\scl}(L)(\xi))^{-1}$ is $\calC^\infty$, its Fourier transform decays rapidly. Indeed, interpreted again
as Fourier transforms of tempered distributions, this follows from the (distributional) identity
\[
\begin{aligned}
\olH_\zeta(w) & = |w|^{-2\ell} (2\pi)^{-n}\int_{\RR^n}  \left(\Delta_\xi^\ell e^{iw\cdot \xi}\right) (\zeta I - \sigma_m^{\scl}(L)(\xi))^{-1}\, d\xi \\ 
& =  |w|^{-2\ell} (2\pi)^{-n}\int_{\RR^n}  e^{iw\cdot \xi}  \Delta_\xi^\ell (\zeta I - \sigma_m^{\scl}(L)(\xi))^{-1}\, d\xi
\end{aligned}
\]
for any $\ell > 0$, and the fact that the integral is classically convergent once $-n > -m - 2\ell$, i.e., $\ell > (n-m)/2$.  
Next, apply any derivative $\del_w^\beta$ to both sides of this equality, where $\beta$ is a multi-index, and choose $\ell > (n-m + |\beta|)/2$ so that 
the integral is still absolutely convergent, to deduce that $\olH_\zeta(w)$ is smooth away from $w=0$.  For an even more refined statement,
apply $w^\alpha \del_w^\beta$ where $|\alpha|= |\beta|$. Passing through the Fourier transform, this becomes a constant multiple
of $\del_\xi^\alpha \xi^\beta$ acting on the exponential, which integrates by parts to an expression where $\xi^\beta (-\del_\xi)^\alpha$ acts on
$\Delta_\xi^\ell (\zeta I - \sigma_m^{\scl}(L)(\xi))^{-1}$.  The resulting integral is once again convergent if $\ell > (n-m)/2$. 
This shows that $\olH_\zeta(w)$ has `stable regularity' with respect to repeated differentiation by the vector fields $w_i \del_{w_j}$, a property
which is known as conormality with respect to the origin $\{w=0\}$.  By a further analysis, which is left to the reader, the expansion
of $\Delta_\xi^\ell (\zeta I - \sigma_m^{\scl}(L)(\xi))^{-1}$ as $|\xi| \to \infty$ is transformed to an expansion in powers of $w$ as $w \to 0$,
which is the assertion that $\olH_\zeta(w)$ is polyhomogeneous at $w=0$. 

We now explain the assertion about exponential decay.  The integrand is a smooth function of $\xi$.   By the uniformity of $L$ as
a function of $\tilde{z}$, there exists some $\delta' > 0$ sufficiently small depending on $\mathrm{Re}\, \zeta < 0$ so that if $\eta \in \RR^n$ 
and $|\eta| < \delta'$, then the deformed map $\xi \mapsto (\zeta I - \sigma_m^{\scl}(L)(\xi + i\eta))^{-1}$ remains smooth in $\xi$ 
and decays as $\xi \to \infty$ at the same rate as the undeformed function.  

Using these observations and a similar renormalization scheme to make the integral convergent, we can deform the contour of integration from $\RR^n$ 
to $\RR^n + i\eta$ and write
\[
\olH_\zeta(w) = e^{- w \cdot \eta} \frac{1}{(2\pi)^n} \int_{\RR^n}  e^{iw \cdot \xi} (\zeta I - \sigma_m^{\scl}(L)(\xi + i\eta))^{-1}\, d\xi.
\]
The integral is defined as before, using the same observations and calculations as above, and is bounded as $w \to \infty$, 
but is now accompanied by the prefactor $e^{- w\cdot \eta}$ which decays exponentially in any cone properly contained in the half-space $w \cdot \eta \geq 0$.
Since $\eta$ can point in any direction, this shows that $\olH_\zeta(w)$ decays exponentially, at some rate depending on $\mathrm{Re}\, \zeta < 0$,
as $|w| \to \infty$.  This solution is smooth in the parameters $\zeta$ and $\tilde{z}$. This establishes the lemma.
\end{proof}

Using these results, it is now straightforward to see that $\olH_\zeta \star f (w)$ is exponentially decreasing in $w$ if $f \in \calC^\infty_0(\RR^n)$.

\medskip

\medskip
\noindent{\bf Manifolds with corners, blowups and polyhomogeneity.}
To get further into the details, we first recall the following terminology and notions.  All of these are described more
carefully in \cite{MazzeoEdge}. First, suppose that $X$ is a manifold with corners. This means
that near any point $q \in X$ there is a local coordinate system $(x_1, \ldots, x_k, y_1, \ldots, y_{\ell-k})$, $\ell = \dim X$, such
that each $x_i \geq 0$ and each $y_i \in (-\epsilon, \epsilon)$, and $(x,y) = (0,0)$ at $q$.  We say that $q$ then lies on a corner
of codimension $k$.    We next define the space $\calV_b(X)$ of $b$-vector fields on $X$ to consist of all smooth vector fields
which are unconstrained in the interior of $X$ and which are tangent to all boundary faces. In terms of the coordinates above,
any $V \in \calV_b(X)$ is, near $q$, a linear combination, with smooth coefficients, of the basic vector fields 
$x_i \del_{x_j}$, $i, j = 1, \ldots, k$ and $\del_{y_i}$, $i = 1, \ldots, \ell-k$. 

\begin{definition}
\label{def:conormal}
A distribution $u$ on $X$ is said to be conormal to $\del X$ if there exists some fixed function space $E$ such that
\[
V_1 \ldots V_N u \in E,\ \ \text{for all}\ V_i \in \calV_b(X)\ \text{and all}\ N \in \mathbb N.
\]

\end{definition}

Typically we let $E$ be a weighted $L^2$ or $L^\infty$ space. Thus if $u$ is conormal, then $u$ is $\calC^\infty$ on the 
interior of $X$, but may have singular behavior along the boundary. These singularities are, however, `tangentially smooth'.

\begin{definition}
\label{def:polyhom}

A distribution $u$ on $X$ is said to be polyhomogeneous at $\del X$ if it is conormal, and in addition has a classical expansion at each boundary face
and product-type expansions at the corners. Thus, near a codimension one face $x=0$, for example, with $y$ a coordinate
system along that face, 
\[
u(x,y) \sim \sum_{\mathrm{Re} \gamma_j \to \infty}  \sum_{k=0}^{N_j}  a_{jk}(y) x^{\gamma_j} (\log x)^k,
\]
for some discrete set of exponents $\gamma_j$, where each coefficient $a_{jk}$ is smooth. Near each corner, where $x_1, \ldots, x_k = 0$, 
$u$ has an expansion involving various (possibly nonintegral) powers of each $x_j$ and, for each monomial in the series, additional factors which are
positive integer powers of each $\log x_i$.

\end{definition}

These expansions hold for all derivatives of $u$ as well, in the sense that any derivative of $u$ has an expansion
where the summands are the corresponding derivatives of each term in the series for $u$.   These series are asymptotic,
but (usually) not convergent.

The space of all such polyhomogeneous distributions is denoted $\calA_{\phg}(X, \del X)$.  We may easily extend this
to define polyhomogeneous sections of vector bundles over $X$ as well.

A submanifold $Z \subset X$ is called a \emph{$p$-submanifold} if near any $q \in Z$ there is a set of coordinates
for $X$ as above such that in some neighborhood of $q$, $Z = \{x_1 = \ldots = x_r = 0,\ y_1 = \ldots = y_s = 0\}$. Thus locally
near $q$, $X$ is the product of (some small neighborhood of) $Z$ with another manifold with corners; the $p$ stands for `product'. 
If $Z$ is such a $p$-submanifold, then we may define the blowup of $X$ around $Z$, denoted $[X; Z]$ to be the union of 
$X \setminus Z$ and the inward-pointing spherical normal bundle of $Z$. Thus $[X ; Z]$ is a new manifold with corners, with
a new boundary hypersurface created by this blowup. As in our special case of the semiclassical double space defined earlier,
we may `construct' this blowup by taking cylindrical coordinates around $Z$ in $X$, say $(r, \omega, z)$, where $z \in Z$,
$\omega$ is a spherical normal vector and $r \geq 0$, and ``adding the $r=0$ face''.  We denote the new `front face'
of this blowup by $\ff [X;Z]$.  

\medskip

\noindent{\bf Pseudodifferential operators via their Schwartz kernels}
Now let $M$ be a manifold (possibly open, but without boundary). We may define the blowup $[M^2, \mathrm{diag}]$.
This has a front face, $\ff$, which is the spherical normal bundle of $\mathrm{diag}$ in $M^2$.
\begin{definition}
A pseudodifferential operator $A:C_c^{\infty}(M) \to \mathcal{D}^{\prime}(M)$ on $M$ is a linear operator for which the Schwartz kernel $K_A$ of $A$ is the sum
of an element of $\calA_{\phg}( [M^2, \mathrm{diag}], \ff)$ and a distribution which is conormal and supported 
along $\ff$. 
\end{definition}
This is a purely intrinsic and `geometric' definition of the space of pseudodifferential operators. It is not easy to
work with computationally, however, and it is more customary to use other definitions using oscillatory integrals,
cf.\ \cite{Shubin} for example, which are better suited for proving such things as showing that the composition of
two pseudodifferential operators is again pseudodifferential. 

The inclusion in this definition of the extra terms which are supported on the front face may seem a technical
annoyance, but it is worth pointing out that the Schwartz kernel of the identity operator, $\delta(z - \tilde{z})$, 
has this property. In fact, if $P$ is any differential operator (with smooth coefficients) on $M$, then
the Schwartz kernel of $P$ is equal to $P$ (acting in the $z$ variable) applied to $\delta(z - \tilde{z})$, and hence
is again supported on this front face.

\medskip

\noindent{\bf Semiclassical pseudodifferential operators}
We are now in a position to define semiclassical families of pseudodifferential operators on $M$.   The semiclassical
double-space $M^2_{\scl}$ has a distinguished submanifold $\mathrm{diag}_{\scl}$, which we call the lifted
diagonal.  It is the closure in $M^2_{\scl}$ of $(0,\vep_0) \times \mathrm{diag}$.  This closure
intersects the front face of $M^2_{\scl}$ at a submanifold which contains a single point $\{w=0\}$ on each
hemisphere fiber.  Clearly $\mathrm{diag}_{\scl}$ is a  $p$-submanifold of $M^2_{\scl}$, so we may pass to
the blowup $[M^2_{\scl}, \mathrm{diag}_{\scl}]$. This has three
boundary hypersurfaces:  the original face at $\vep=0$ away from the diagonal, the front face $\ff$ obtained
by blowing up the diagonal at $\vep=0$, and the new front face obtained in this final blowup.

We say that $A_\vep$ is a \emph{semiclassical pseudodifferential operator} \label{ref:sc-psido} if its Schwartz kernel $K_A(\vep, z, \tilde{z})$ lifts
to $[M^2_{\scl}, \mathrm{diag}_{\scl}]$ to be the sum of a distribution polyhomogeneous at all boundaries of this manifold with corners and
a conormal distribution supported on the new front face, and if $K_A$ vanishes faster than any power of $\vep$ along the
original face. It makes sense to restrict $K_A$ to each level set $\pi^{-1}(\vep)$ for $\vep > 0$,
and on each of these we must obtain the Schwartz kernel of a pseudodifferential operator.  This definition
now imposes precise constraints on how these Schwartz kernels degenerate as $\vep \to 0$.  In particular,
if we change to coordinates $\vep, w, \tilde{z}$ on $M^2_{\scl}$, then the `pseudodifferential singularity' occurs
at $w=0$ for each $\vep$, and at $\vep = 0$ there is an expansion $K_A \sim \sum \vep^{\gamma_j} K_A^{(j)}(w,\tilde{z})$
where each coefficient is a pseudodifferential operator on each hemisphere fiber of $\ff$ (and as before, $\{\gamma_j\}$
is a discrete set of exponents with real parts tending to infinity). 

In fact, we are only interested in Schwartz kernels for which the expansion as $\vep \to 0$ is of the form
\[
K_A(\vep, w, \tilde{z}) \sim \sum_{j=-n}^\infty  \vep^j K_{A,j}(w,\tilde{z}).
\]
The reason for starting this series at $-n$ is as follows.  As we have already described, the lift of $\vep^m L$ (as a differential
operator, not its Schwartz kernel) is an operator with coefficients smooth up to $\ff$; in particular, it has a series expansion
as $\vep \to 0$ with leading term $\sigma_m^{\scl}(L)(\del_w)$.  On the other hand, the identity operator has Schwartz kernel
\[
\delta(z - \tilde{z}) = \delta( \vep w) = \vep^{-n} \delta(w).
\]
Thus if we expand each factor in the equation
\begin{equation}
(\zeta I - \vep^m L) G(\vep, w, \tilde{z}) = \vep^{-n} \delta(w)
\label{expand}
\end{equation}
in powers of $\vep$, and continue to think of the first factor on the left as a differential operator instead of a Schwartz kernel,
then at least formally we expect $G$ to have a series expansion involving the powers $\vep^j$ for $j \geq -n$.  

In any case, this illustrates how the introduction of the space $M^2_{\scl}$ provides a setting where it is possible to ``see'' the
transition from the ordinary pseudodifferential operators on each $\pi^{-1}(\vep)$ to the limiting model problem on $\RR^n$.

Accordingly, we now define a class of semiclassical pseudodifferential operators whose expansions at $\ff$ involve only integer powers: 
\begin{definition} 
\label{def:sc-pseud}
If $(M,g)$ is a manifold of bounded geometry, then $\Psi^{k,\ell}_{\scl-\unif}(M,g)$ consists of those semiclassical pseudodifferential
operators on $M$ which have pseudodifferential order $k$ on each level set $\pi^{-1}(\vep)$ and which have a series expansion in $\vep$
at $\ff$ with only integer powers, with initial term $\vep^\ell$. 

The subscript `unif' indicates that we restrict further to allow only kernels which are supported in some neighborhood 
$\mathrm{dist}_g(z, \tilde{z}) \leq C$ of $\mathrm{diag}_{\scl} \cup \ff$, where the constant $C$ may depend on the operator. 
\end{definition}

\medskip

\noindent{\bf Parametrix construction}
We now commence with the parametrix construction.  Our goal is to find an element $G \in \Psi^{-m,-n}_{\scl-\unif}(M,g)$ such
that, for each $\zeta$ with $|\zeta| = 1$ and $\mathrm{Re}  \, \zeta < 0$, 
\[
(\zeta I - \vep^m L) G = I - Q_{1}, \quad G (\zeta I - \vep^m L) = I - Q_{2},
\]
where $Q_1, Q_2 \in \Psi^{-\infty, \infty}_{\scl-\unif}(M,g)$.  As explained above, we do this `formally', i.e., in Taylor series, at $\ff$,
and then take a Borel sum of the resulting formal expansion.

More carefully, returning to \eqref{expand}, expand each factor into its formal series expansion:
\[
\left( (\zeta I - \sigma_m^{\scl}(L)(\del_w)) + \sum_{j=1}^\infty \vep^j E_j \right) \sum_{k=-n}^\infty \vep^k G_k(w,\tilde{z}) = \vep^{-n} \delta(w).
\]
Here the $E_j$ are the differential operators (of order $\leq m$) arising in the Taylor expansion of $\zeta I - \vep^m L$. Carrying out the 
composition and collecting like powers of $\vep$, we obtain a sequence of equations
\[
\begin{aligned}
(\zeta I - \sigma_m^{\scl}(L)(\del_w)) G_{-n}(w, \tilde{z}) & = \delta(w) \\
(\zeta I - \sigma_m^{\scl}(L)(\del_w)) G_{-n+k}(w, \tilde{z}) & = F_k(w,\tilde{z}),
\end{aligned}
\]
where each $F_k = \sum_{i=1}^{k} E_i G_{-n+ k-i}$ is an `error term'.   This is a sequence of equations on $\RR^n_w$, where the various
terms all depend smoothly on $\tilde{z}$.  These equations indicate that we must set
\[
G_{-n}(w) = \olH_\zeta(w), \quad \mbox{and}\ \ \ G_{-n+k}(w) = \olH_\zeta(w) \star F_k(w)\ \ k \geq 1. 
\]
for each $\tilde{z}$. Since the equations with $k \geq 1$ are constant coefficient in $w$, they are solved by convolving with
the fundamental solution $\olH_\zeta(w)$.

Since each $E_i G_{-n+k-i}$ has pseudodifferential order at most $0$ and $\olH_\zeta$ has order $-m$, each $F_k$ is the Schwartz kernel
of a translation-invariant (in $w$) pseudodifferential operator of order no more than $-m$.  We may solve this equation
inductively, given the properties  of $\olH_\zeta$ established in Lemma \ref{lem:greenfunc}, but then are faced with showing that 
the Borel sum of this series of singular kernels still has the correct behavior. There is a slightly easier way to proceed which allows
us to work more directly with $\calC^\infty$ kernels.  Namely, we first find a pseudodifferential operator $\widetilde{G}(\vep, \cdot, \cdot)$ 
on each level set $\pi^{-1}(\vep)$ which solves 
\[
(\zeta I - \vep^m L) \widetilde{G}(\vep) = \vep^{-n} \delta(w) + \vep^{-n}\calR(\vep, w, \tilde{z}),
\]
where $\calR$ is smooth in all variables, $\vep, w, \tilde{z}$. This involves carrying out the complete parametrix construction
for the nondegenerate operator on each level set in the standard pseudodifferential calculus, but carrying along $\vep$ as a smooth parameter.
To compensate for the factors $\vep^{-n}$ on the right, we choose $\widetilde{G} \in \Psi^{-m, -n}_{\scl-\unif}$.  Said differently, this
is simply the nondegenerate elliptic parametrix construction (with parametrix cut off to have support in a neighborhood of the diagonal),
carried out smoothly in the parameter $\vep$.

We now need to find additional terms in the parametrix which cancel off the full Taylor series in $\vep$ of the 
remainder term $\vep^{-n}\calR$. This involves inductively solving a sequence of equations
\[
\begin{aligned}
(\zeta I - \sigma_m^{\scl}(L)(\del_w)) \widetilde{G}_{-n}(w, \tilde{z}) & = \calR_0(w,\tilde{z}) \\
(\zeta I - \sigma_m^{\scl}(L)(\del_w)) \widetilde{G}_{-n+k}(w, \tilde{z}) & = \sum_{i=1}^k  E_i \calR_{-n + k-i},
\end{aligned}
\]
where $\calR \sim \sum \vep^k \calR_k$.  The advantage is that the right hand sides are all smooth and we can assume that 
they are compactly supported in, say, $\{|w| \leq 1\}$ for all $\tilde{z}$.  The solutions are given by
\[
\widetilde{G}_{-n} = \olH_\zeta\star \calR_0 (w), \qquad \widetilde{G}_{-n+k} = \olH_\zeta \star \left( \sum_{i=1}^k E_k \calR_{-n+k-i}\right)(w).
\]

\noindent{\bf Conclusion of parametrix construction}  We have now constructed both $\widetilde{G}$ and the sequence
of smooth, exponentially decaying terms $\widetilde{G}_{-n+k}$, $k \geq 0$.  An important fact is that given any such
sequence, it is possible to construct a function $\widetilde{G}'(w)$ which is smooth in the interior of $M^2_{\scl}$, decays to all 
orders at the original face, and which has expansion in powers of $\vep$ at $\ff$ of the form
\[
\widetilde{G}'(w) \sim \sum_{k=0}^\infty \vep^{-n+k} \widetilde{G}_k(w).
\]
This is the Borel sum of this series and is an element of $\Psi^{-\infty, -n}_{\scl-\unif}$.

Our final semiclassical parametrix is now defined by
\[
G(\vep, w, \tilde{z}) = \widetilde{G}(\vep, w, \tilde{z}) + \widetilde{G}'(\vep, w, \tilde{z}) \in \Psi^{-m, -n}_{\scl-\unif}.
\]

\medskip
\noindent{\bf Boundedness properties}
We conclude this section by sketching the proof of boundedness properties of elements of $\Psi^{*,*}_{\scl-\unif}(M)$ on the
semiclassical H\"older spaces $\calC^{k,\alpha}_\scl(M,g)$, as defined in \S \ref{sec:background}.   First define the family of spaces
\[
\vep^\lambda \calC^{k,\alpha}_{\vep} = \{u(\vep, z) = \vep^\lambda \tilde{u}(\vep, z)\ \mbox{where}\ \tilde{u} \in \calC^\infty( \, [0, \vep_0)\, ; \, 
\calC^{k,\alpha}_{\vep}(M,g) \, )\}.
\]
In other words, an element of this space can be represented by a formal series $u \sim \sum_{j \geq 0} \vep^{\lambda+j} u_j$ with
each $u_j \in \calC^{k,\alpha}_{\vep}$. 
\begin{proposition}
If $A \in \Psi^{\kappa, \mu}_{\scl-\unif}$ for some $\kappa \in \mathbb Z$, then
\[
A: \vep^\lambda \calC^{k,\alpha}_{\vep}  \longrightarrow \vep^{\lambda + \mu + n}\calC^{k-\kappa,\alpha}_{\vep}
\]
is bounded provided $\kappa \leq k$. 
\label{bddprop}
\end{proposition}
\begin{remark}
We can extend this result to allow operators of non-integral order, for example using a standard interpolation result, but
omit this here since it is not needed.
\end{remark}
\begin{proof}
The first observation is that if $\kappa \leq k$, and $P = \sum_{|\alpha| \leq \kappa} a_\alpha(z) \vep^{|\alpha|}\del_z^\alpha$ is a semiclassical 
{\it differential} operator with coefficients which are uniformly bounded in $\calC^\infty$, then directly from the definition,
\begin{equation}
P: \calC^{k,\alpha}_\vep(M,g) \longrightarrow \calC^{k-\kappa,\alpha}_{\vep}(M,g)
\label{bddP}
\end{equation}
is bounded.  To relate this to a pseudodifferential boundedness theorem, we have previously noted that the Schwartz kernel of $P$ on $M^2_{\scl}$
is a distribution supported along $\mathrm{diag}_{\scl}$, namely
\[
K_{P}(\vep, w, \tilde{z}) = \vep^{-n} \sum_{|\alpha|\leq \kappa}  a_\alpha(\tilde{z} + \vep w) (\del_w^\alpha \delta(w)) \in \Psi^{\kappa, -n}_{\scl-\unif}.
\]
Thus \eqref{bddP} is a very special case of the boundedness in Proposition \ref{bddprop}, with $\lambda = 0$.

Recalling the definition of a semiclassical pseudodifferential operator from page \pageref{ref:sc-psido}, we prove the boundedness property for kernels in $\Psi^{-\infty,-n}_{\scl-\unif}(M,g)$ supported away from the diagonal, and another for kernels in $A \in \Psi^{\kappa,-n}_{\scl-\unif}(M,g)$ supported near the diagonal.

Thus first suppose that $A \in \Psi^{-\infty,-n}_{\scl-\unif}(M,g)$ has Schwartz kernel $K_A$ supported away from $\mathrm{diag}_{\scl}$ in $M^2_{\scl}$; 
to be specific, suppose $\mathrm{supp}(K_A) \subset \{\vep \leq \mathrm{dist}(z, \tilde{z}) \leq C\}$, or equivalently, $1 \leq |w| \leq C/\vep$.   
We then have that in projective coordinates $(z, \tilde{w})$, $\tilde{w} = (\tilde{z} - z)/\vep$, 
\begin{align*}
\left|\int_M K_A(\vep, z, \tilde{z}) u(\tilde{z})\, dV_g(\tilde{z})\right|  &\leq C \|u\|_{\infty} \int_{1 \leq |\tilde{w}| \leq C/\vep} |K_A(\vep, z, \tilde{w})|\, \vep^n dV(\tilde{w}) \\ 
&\leq C \|u\|_{\infty} \int_{1 \leq |\tilde{w}| \leq C/\vep} ( 1 + |\tilde{w}|)^{-N}\, d\tilde{w} \leq C' \|u\|_{\infty} < \infty
\end{align*}
since $K_A$ blows up like $\vep^{-n}$ and $\vep^n |K_A|$ is Schwartz in $\tilde{w}$, uniformly as $\vep \searrow 0$.   Furthermore,
any semiclassical derivative $\vep^{|\alpha|} \del_z^\alpha$ applied to $K_A$ yields a kernel of the same form. This proves that 
if $A \in \Psi^{-\infty, -n}_{\scl-\unif}$ has kernel with this special support property, then
\[
A: \calC^{k,\alpha}_{\vep}(M,g) \longrightarrow \calC^{r,\alpha}_{\vep}(M,g)
\]
for any $r \in \mathbb N$.

Now let $A \in \Psi^{\kappa,-n}_{\scl-\unif}(M,g)$ have Schwartz kernel supported in the region $|w| \leq 1$, and as before, assume $\kappa \leq k$. 
We do not need to assume that $\kappa$ is an integer, but then must change the H\"older indices accordingly and interpret the borderline
cases where $\kappa = k -i + \alpha$, $i \in \mathbb N$, in terms of Zygmund spaces.
We can then proceed by a combination of a rescaling argument and invoking the fact that an ordinary pseudodifferential operator of order $\kappa$ 
in a ball of size $2$ induces a map $\calC^{k,\alpha}_0(B_1(0)) \rightarrow \calC^{k-\kappa,\alpha}(B_1(0))$, between {\it ordinary} H\"older spaces
(and where elements of the domain space are compactly supported in $B_1(0)$).  Indeed, given any $u \in \calC^{k,\alpha}_{\vep}$, choose a locally
finite cover $B_{2\vep}(q_j)$ of bounded covering multiplicity and a partition of unity $\chi_j$ for which there are uniform bounds on its
semiclassical derivatives up to order $k$, and such that $\chi_j = 1$ on $B_{\vep}(q_j)$.  Then $u = \sum \chi_j u$ and since only at most some fixed number of the $A(\chi_j u)$ have support at any one point, it suffices to estimate the norms of each of these summands. 
Recalling Remark~\ref{scnormballs}, we may compute the semiclassical H\"older norms by taking the supremum over balls
of radius $\vep$ (or any fixed multiple of $\vep$).  Thus, rescaling the coordinates of each summand by a factor of $1/\vep$, we reduce to the action of a standard pseudodifferential operator of order $\kappa$ on a standard $\calC^{k,\alpha}$ function on a ball of fixed radius, 
where the result is well known.

Finally, if $A \in \Psi^{\kappa, \mu}_{\scl-\unif}(M,g)$, then it is clear from all of the above, and the basic definitions, that
the assertion of Proposition \ref{bddprop} holds.
\end{proof}

\medskip

\noindent{\bf Operators with finite regularity coefficients.}  
As explained at the beginning of this section, the geometric microlocal techniques used here require the smoothness 
of both the metric $g$ and the coefficients of $L$. Using that assumption, we have shown that the inverse 
$(\zeta I - \vep^m L)^{-1}$ exists and is bounded on $\calC^{k,\alpha}$ if $\vep$ is sufficiently small.  We now
extend this to operators and metrics of lower regularity. 

\begin{prop}
Suppose that $(M,g)$ has bounded geometry of order $\ell + \alpha'$ where $\ell \geq k + m$ and $\alpha' \in (\alpha, 1)$, 
and that $L$ is an admissible elliptic operator of order $m$ with coefficients uniformly bounded in $\calC^{k,\alpha'}$.  
Then there exists $\vep_0 > 0$ such that $G_{\zeta, \vep} := (\zeta I - \vep^m L)^{-1}$ exists as a bounded operator on 
$\calC^{k,\alpha}$ if $0 < \vep < \vep_0$.
\label{finitereg}
\end{prop}

While it is possible to carry out some version of the parametrix construction under these regularity hypotheses, this would
take extra work, particularly when the regularity order $k$ is small. Thus we prove this another way, using techniques
close to those in Section \ref{sec:proofThm}. 

\begin{proof}
Choose an approximating sequence of metrics $g^{(i)}$ and operators $L^{(i)}$, which are all smooth, but so that 
$g^{(i)} \to g$ in $\calC^{\ell, \alpha'}$ and the coefficients of $L^{(i)}$ converge to those of $L$ in $\calC^{k,\alpha'}$.
This can be done by selecting a locally finite open cover of $M$ by normal coordinate balls for $g$ and using a mollifier
in each such ball. 

Although the norms on the spaces $\calC^{j,\beta}$ depend on the metric, it is clear that we may define these
all relative to any fixed smooth metric. In fact, we assume that the metric $g$ is fixed (and smooth) for simplicity,
since the way it enters the argument below is minor. 

As has been shown earlier in this section, for each $i$, and for $0 < \vep < \vep_0^{(i)}$,  there exists a bounded inverse
\[
G^{(i)}_{\vep,\zeta} = (\zeta I - \vep^m L^{(i)})^{-1}: \calC^{k,\alpha} \longrightarrow \calC^{k,\alpha}.
\]
Denote the operator norm of this inverse by $A_{\vep,\zeta, i}$. There are two main points we must address. 
The first is that there exists $\vep_0$ where $\vep_0^{(i)} \geq \vep_0 > 0$ for all $i$, and the second is that the operator norms 
$A_{\vep,\zeta,i}$ are bounded for each $\vep$ and $\zeta$ as $i \to \infty$. 

Suppose first then that there exist sequences $\vep^{(i)} \to 0$ and $\zeta^{(i)}$ such that 
\[
P_i  := (\zeta^{(i)} I - (\vep^{(i)})^m L^{(i)})
\]
does not have a bounded inverse.  This failure occurs for one of three reasons: either $P_i$ has nullspace in $\calC^{k,\alpha}$, 
or its range is dense but not closed, or the closure of its range has positive codimension. All of this is just as in Proposition~\ref{limspecrel},
and the proofs to rule out each of these cases is very similar to the proof of that Proposition, as well as the arguments of Section \ref{sec:proofThm}.
For that reason, we shall be brief. 

In the first of these cases, there exists a sequence $u^{(i)} \in \calC^{k,\alpha}$ such that $||u^{(i)}||_{k,\alpha} = 1$
and $P_i u^{(i)} = 0$.  We may extract a limit $u$ of this sequence by rescaling around a point $q_i \in M$ where $|u^{(i)}(q_i)| \geq 1/2$.
Since $\vep^{(i)} \to 0$, this limiting function $u$ lies in $\calC^{k,\alpha}(\RR^n)$ and satisfies $(\zeta I - L_E) u = 0$, where $\zeta$
is a limit of (a subsequence of the) $\zeta^{(i)}$, and $L_E$ is the constant coefficient strongly elliptic operator arising in
this rescaling process. As showed in the proof of Proposition~\ref{prop3.3}, by the strong ellipticity of $L$ (and hence $L_E$), 
there are no nontrivial solutions of this equation. 

Next, if the range of $P_i$ is dense, there exist sequences $u^{(i)}, f^{(i)} \in \calC^{k,\alpha}$ such that $P_i u^{(i)} = f^{(i)}$, 
with $||u^{(i)}||_{k,\alpha} = 1$ and $||f^{(i)}||_{k,\alpha} \to 0$. Just as in the previous paragraph, there is a nontrivial limit
$u \in \calC^{k,\alpha}(\RR^n)$ such that $(\zeta I - L_E)u = 0$, which is impossible. 

Finally, if the closure of the range of $P_i$ is a proper subspace, then we may apply the same type of argument to the sequence
of distributions $v^{(i)}\in (\calC^{k,\alpha})^*$ which satisfy $P_i^* v^{(i)} = 0$. To do this, we must note that since $v^{(i)}$ 
satisfies this elliptic equation, it lies in $\calC^{k+m,\alpha}$. There is a limiting function $v$ which satisfies
$(\bar{\zeta} I - L_E^*)v = 0$, which again cannot happen.  This proves that the inverses $G_{\vep,\zeta}^{(i)}$ all exist for $\vep$ 
lying in some fixed interval $(0, \vep_0)$. 

Now fix any $\vep$ in this interval, and any $\zeta$, and suppose that the norms $A_i = ||G_{\vep,\zeta}^{(i)}||_{\mathcal L(\calC^{k,\alpha})}$ 
are unbounded as $i \to \infty$. This implies that there is a sequence $f_i \in \calC^{k,\alpha}$ such that 
\[
||G_{\vep, \zeta}^{(i)} f_i||_{k,\alpha} \geq \frac12 A_i ||f_i||_{k,\alpha}.
\]
Writing $u_i = G^{(i)}_{\vep,\zeta} f_i$, this is the same as
\[
||u_i||_{k,\alpha} \geq \frac12 A_i || (\zeta_i I - \vep^m L^{(i)}) u_i||_{k,\alpha}.
\]
Normalizing so that $||u_i||_{k,\alpha} = 1$ for all $i$, then $||f_i||_{k,\alpha} \leq 2/A_i \to 0$.  Passing to a limit as usual,
but recalling that $\vep$ is fixed, there exists a limiting function in $\calC^{k+m,\alpha}$,  defined either on $M$ or on one of its
limiting spaces $M_\infty$, such that $(\zeta I - \vep^m L)u = 0$ or $(\zeta I - \vep^m L_\infty)u_\infty = 0$. 

The proof is complete once we show that these last possibilities cannot occur.   Let us focus on the first, since the second
is essentially the same.  The point is simply that $\zeta I - \vep^m L$ cannot have nullspace for arbitrarily small
values of $\vep$. Indeed, if this operator were to have nullspace in $\calC^{k,\alpha}$ for some sequence $\vep_i\to 0$,
then the same rescaling argument as we have done several times already would yield a limiting function $u$ in
$\calC^{k,\alpha}$ on $\RR^n$ such that $(\zeta I - L_E)u = 0$, which we know is impossible. 
\end{proof}
\begin{remark}
It is perhaps worth emphasizing the flow of logic in this argument. We first show that for every one of the approximating 
operators $L^{(i)}$, the operator $\zeta I - \vep^m L^{(i)}$ is invertible for $\vep < \vep_0$ where $\vep_0$ does not depend
on $i$. However, it may be necessary to restrict to a slightly smaller interval $0 < \vep < \vep_1 < \vep_0$ in order to
guarantee that $\zeta I - \vep^m L$ does not have nullspace.  

The main point is that it is only the constant coefficient operators $\zeta I - L_E$ which we can check specifically do not have nullspace.
These operators only appear as limits when $\vep \to 0$, and the argument to rule out nullspace when $\vep$ is sufficiently small,
while it is not quantitative, is insensitive to the regularity of the coefficients of $L$. 
\end{remark}

\section{Applications}
\label{sec:applicats}
In this section we present an application of the results proven in this paper.  We begin by stating a general theorem which establishes short-time existence,
uniqueness and continuous dependence on initial conditions for a large class of geometric flows on manifolds with bounded geometry.
We then illustrate its application by obtaining a new result concerning the short-time existence and stability of the higher order `ambient 
obstruction flow' on open manifolds with bounded geometry.

\subsection{A general result}
For $k, m \in \bN$, $0 < \alpha < \alpha' < 1$, suppose that $(M,g)$ is a complete Riemannian manifold with bounded geometry of order $k + m + \alpha'$, and let $F$ be a smooth,
possibly nonlinear, elliptic partial differential operator of order $m$ acting on some open subset of the space of sections of some 
uniform vector bundle over $M$. Set $X = \calC^{k,\alpha}(M,g)$ and $D = \calC^{m+k,\alpha}(M,g) \subset X$. (As above, these are little 
H\"older spaces.)  Let $\mathcal U$ be an open subset of $D$ for which 
\[
F: \mathcal U \longrightarrow  X
\]
is a smooth mapping.  Now consider the Cauchy problem for  Banach-valued sections:
\begin{equation} 
\frac{du}{dt}  = F( u(t) ), \qquad u(0) = u_0 \in \mathcal U. 
\label{eqn:CPforu}
\end{equation}
Assume that the linearization $DF_u$ at any $u \in \mathcal U$ is admissible; hence by Theorem \ref{thm:main-A}, each such 
$DF_u$ is sectorial as an unbounded map from $X$ to itself. Following \cite{Lunardi}, we then obtain the wellposedness result of Theorem \ref{thm:main-B}, restated here for convenience:
\begin{theorem} \label{thm:STE}
Assuming the notation as well as the hypotheses above,
\begin{enumerate}
\item (\emph{Short-time existence, uniqueness})
There exists $T > 0$ such that the initial value problem \eqref{eqn:CPforu} has a unique smooth solution for $t \in [0,T)$. 
\item (\emph{Continuous dependence}) In addition, the estimate
\begin{equation}
\label{wplocest}
 \|v(t) - w(t)\|_{D} \leq C \| v_0 - w_0\|_D, \; \; \mbox{for all} \; t \in [0,\vep)
\end{equation}
is valid for any two solutions $v(t)$ and $w(t)$ with initial values $v_0, w_0 \in \mathcal U$.
\end{enumerate}
\end{theorem}
\begin{proof}
Observe that since we are using little H\"older spaces, $D$ is dense in $X$.  Note also that since sectoriality is an open condition,
see \cite[Proposition 2.4.2]{Lunardi}, it suffices to prove sectoriality of $L = DF_{u_0}$ at any one particular point $u_0 \in
\mathcal U$ in order to prove sectoriality at every $u'_0$ in some perhaps slightly smaller neighborhood $\mathcal U'$. 
Invoking Theorem \ref{thm:main-A}, we may now apply Theorem 8.1.1 and Corollary 8.1.2 in \cite{Lunardi} to obtain
existence and uniqueness; continuous dependence in $t$ is addressed in Section 8.3 in \cite{Lunardi}. 
\end{proof}

 \subsection{Ambient obstruction flows}
In their study of conformal invariants of a compact manifold endowed with a conformal structure, Fefferman and Graham introduce 
the ambient obstruction tensor, \cite{FG}.  If $n=2\ell$ is even, the ambient obstruction tensor $\mathcal{O}_n$ on a manifold $(M^n, g)$
is a conformally covariant, trace-free, divergence-free symmetric $2$-tensor associated to the metric $g$. Its expression involves $n-2$ derivatives 
of the Ricci tensor. In the particular case $n=4$, the obstruction tensor $\mathcal{O}_4$ coincides with the Bach tensor
\begin{align}
B_{ij} = {P_{ij,k}}^k - {P_{ik,j}}^k - P^{kl} \, W_{kijl},
\end{align}
where $P_{ij} = \frac{1}{2} \left( Rc_{ij} - \frac{S}{6} g_{ij} \right)$, and $W_{ijkl}$ and $S$ are the Schouten tensor, the Weyl tensor and the 
scalar curvature of $g$, respectively. We refer to \cite{FG}, where the importance of this tensor to conformal geometry is explained.

We now study a flow associated to this ambient obstruction tensor. If this flow exists and converges as $t \to \infty$, the limit must 
be ``obstruction-flat'', a condition describing a natural class of canonical metrics in higher dimensions. On compact manifolds, the
wellposedness and uniqueness of solutions to this flow is the topic of the two papers \cite{BahuaudHelliwell, BahuaudHelliwell2} by the first author here and Helliwell.  As an application of the methods of the present paper, we generalize these results to the setting 
of complete manifolds of bounded geometry.  We describe this briefly here and refer the reader to \cite{BahuaudHelliwell, BahuaudHelliwell2} for more detail. 

The obstruction flow itself, namely $\del_t g = \mathcal O_n(g)$, is degenerate both because of the underlying conformal
covariance as well as the usual diffeomorphism invariance. To counter the first of these, we introduce the modified obstruction flow
\begin{align} \label{eqn:AOF}
\begin{cases} \partial_t g &= \mathcal{O}_n(g) + c_n (-1)^{\frac{n}{2}} ( (-\Delta)^{\frac{n}{2}- 1} S ) g \\
g(0) &= g_0,
\end{cases}
\end{align}
where
\begin{align}
c_n = \frac{1}{2^{n/2 - 1} ( \frac{n}{2} - 2)! (n-2) (n-1)}.
\end{align}
In $4$ dimensions this is the modified Bach flow
\begin{align}
\begin{cases} \partial_t g &= B(g) - \frac{1}{12} (\Delta S ) g \\
g(0) &= g_0.
\end{cases}
\label{mBF}
\end{align}
This modification breaks the conformal gauge in the sense that stationary points of this modified flow are obstruction 
flat metrics with harmonic scalar curvature. The scalar curvature condition is the normalization within a conformal class.
The proof of this uses that $\mathcal{O}_n$ is trace-free. 

The invariance under diffeomorphisms can be handled using a version of DeTurck's method, which is as an effective
tool to handle this degeneracy for the Ricci flow (see Chapter 2, Section 6 of \cite{CLN}).   We now describe this method
in the present setting.  Fix a background metric $\gtil$; then any smooth one-parameter family of 
metrics $g(t)$ now defines a time-dependent vector field 
\begin{align}
\label{def:DeTvf}
V(t) = \sum V_k(t,z) \del_{z_k}, \quad \mbox{where}\qquad V^k(t,z) := g^{pq}(t) \left( \Gamma(g(t))^k_{pq} - \Gamma(\gtil)^k_{pq} \right)
\end{align}
using the Christoffel symbols $\Gamma$ of the indicated metrics.  From this we define the DeTurck vector field
\begin{align} \label{eqn:DT-vectorfield}
U = c_n (n-1) (-1)^{\frac{n}{2}-1} (-\Delta)^{\frac{n}{2}-1} V + \frac{c_n ( n-2) (-1)^\frac{n}{2} }{2} (-\Delta)^{\frac{n}{2}-2} \, \nabla S.
\end{align}
and finally the obstruction-DeTurck flow 
\begin{align} \label{eqn:ODT}
\begin{cases} \partial_t g &= \mathcal{O}_n(g) + c_n (-1)^{\frac{n}{2}} ((-\Delta)^{\frac{n}{2}- 1} S ) g + L_U g \\
g(0) &= g_0.
\end{cases}
\end{align}

As usual, one must show that solutions of this gauged flow lead to solutions of the original (modified) flow \eqref{eqn:AOF}. 
To this end, given a solution $g(t)$ to \eqref{eqn:ODT}, solve the ODE
\begin{align} \label{eqn:DT-trick}
\begin{cases}
\frac{d}{dt} \phi_t&=-U \circ \phi_t\\
\phi_0&=\mathrm{id},
\end{cases}
\end{align}
and let $\phi_t$ be the one-parameter family of diffeomorphisms generated by $-U$.   The fact that $\gtil$ and $g(t)$ have bounded geometry 
implies that $\phi_t$ exists at least for $t$ in some small interval around $0$.  A short calculation then shows that $\bar{g}(t) = \phi^*_t g(t)$ 
solves \eqref{eqn:AOF}.

Uniqueness of solutions to the gauged flow \eqref{eqn:ODT} follows directly from the semigroup method that we invoke below.
Uniqueness of solutions to the ungauged flow \eqref{eqn:AOF} requires more work. This is explained carefully in \cite{BahuaudHelliwell2},
but the main ideas are as follows.  Given a particular solution $\overline{g}(t)$ to \eqref{eqn:AOF} and a choice of reference metric $(M,\gtil)$, one may again use semigroup techniques to solve a higher order analogue of the harmonic map heat flow equation for a family of diffeomorphisms $\phi_t$ from $(M,\overline{g}(t))$ to $(M,\gtil)$.  This equation is chosen exactly so that pullback $g(t) = (\phi_t^{-1})^* \gbar(t)$ solves \eqref{eqn:ODT} with reference metric $\gtil$ and $U$ and $V$ as above.  The various uniqueness statements then imply that $\overline{g}$ is uniquely determined.  

We now expand on both the existence and uniqueness statements.  Taking the reference metric $\gtil$ equal to the initial metric, i.e.,\ $\gtil = g_0$, we define
\[ 
F( g ) := \mathcal{O}_n(g) + c_n (-1)^{\frac{n}{2}} ( (-\Delta)^{\frac{n}{2}- 1} S ) g + L_U g. 
\]
As proved in \cite{BahuaudHelliwell}, 
\begin{align}
DF_{g_0} = \left. \frac{d}{ds} F( g_0 + s h) \right|_{s=0} &= (-1)^{\frac{n}{2}-1} A_{g_0} h + \mathcal{P}( \partial^{n-1} h, \partial^n g_0, g_0^{-1},\partial^n \gtil, \gtil^{-1} ),
\end{align}
where the leading term 
\begin{align}
(A_{g_0} h)_{jk} := g_0^{r_1 s_1} g_0^{r_2 s_2} \cdots g_0^{r_{n/2} s_{n/2}} \partial_{r_1} \partial_{s_1} \ldots \partial_{r_{n/2}} \partial_{s_{n/2}} h_{jk}
\end{align}
is an operator of order $n$ and $\mathcal{P}$ is a polynomial expression in the input tensors and their derivatives of appropriate order. Note that 
$A_{g_0}$ is the leading term in $\Delta^{n/2}$ and is strongly elliptic.  We are now ready to prove Theorem \ref{thm:main-C}, restated here for convenience.

\begin{theorem} 
Let $(M^n,g)$ be a complete Riemannian manifold of bounded geometry of order $2n + \alpha'$, with even dimension 
$n = 2\ell$, and where $0 < \alpha < \alpha' < 1$. If $g_0$ is any smooth metric on $M$, then there exists $T > 0$ and a family of unique metrics $g: [0,T) \to \calC^{n,\alpha}(M,g)$ solving the ambient obstruction flow 
\begin{align} 
\begin{cases} \partial_t g &= \mathcal{O}_n(g) + c_n (-1)^{\frac{n}{2}} ( (-\Delta)^{\frac{n}{2}- 1} S ) g \\
g(0) &= g_0.
\end{cases}
\end{align}
\end{theorem}
\begin{proof}
Set $D = \calC^{2n,\alpha} \subset X = \calC^{n,\alpha}$ and observe that $F: D \longrightarrow X$ is smooth. Since $DF_g$ is constructed in terms 
of the metric tensor, the uniform geometry implies that $DF_g$ is admissible. Hence by Theorem \ref{thm:STE}, there is a short-time solution to the 
obstruction-DeTurck flow.

As explained above, the equation \eqref{eqn:DT-trick} can then be solved to obtain the family of diffeomorphisms $\phi_t$,
and we then deduce that $\bar{g}(t) = \phi_t^* g(t)$ is a short-time solution to the obstruction flow with initial condition $g_0$.

To argue uniqueness, suppose that $\overline{g}_i(t)$, $i=1,2$ are two solutions to \eqref{eqn:AOF} with the same initial condition $g_0$.  Again choose reference metric $\gtil = g_0$.  Following Section 5.2 of \cite{BahuaudHelliwell2}, for each $i$ we set
\[
E(\phi_i) := (-1)^{n/2} c \Delta_{\overline{g}_i, g}^{n/2}\phi_i + \mathcal P(\phi_i),
\]
where $\Delta_{\overline{g}_i, g}$ is the Laplacian associated to the `map covariant derivative' for the identity map $(M, \overline{g}_i(t)) \to (M, \gtil)$,
as described in \cite{BahuaudHelliwell2},  and where $\mathcal P$ is a nonlinear differential operator of order $n-1$ acting on $\phi$. 
Combining this with the ODE for $\phi$ itself, we arrive at the strictly parabolic equation
\begin{equation}
\del_t \phi_i = E(\phi_i), \; (\phi_i)(0) = \mathrm{id}.
\label{deq}
\end{equation}

Taking advantage of the explicit coordinate expression for $E$ in \cite{BahuaudHelliwell2}, and using the bounded geometry
of $M$ with respect to either of the metrics $g(t)$ or $\overline{g}(t)$ (valid in some fixed time interval), we see that $DE$ 
is an admissible elliptic operator, and hence Theorem~\ref{thm:STE} may be applied to \eqref{deq} to conclude that this equation has a unique solution on some short time interval that remains a diffeomorphism.  The remainder of the argument finishes exactly as in Section 5.3 of \cite{BahuaudHelliwell2}.
\end{proof}



\begin{thebibliography}{99}

\bibitem{Amann1} H. Amann, \emph{Linear and quasilinear parabolic problems}, Monographs in Math, vol. 89 Birkhauser.
\bibitem{Amann2} H. Amann, \emph{Nonhomogeneous linear and nonlinear elliptic and parabolic boundary value problems}, Funct. spaces, Diff. Operators and Nonlin. Analy, Teubner-Texte Math.,
{\bf 133} (1993), 9-126
\bibitem{Angenent} Clement, Ph.; Heijmans, H.J.A.M. \emph{One-parameter semigroups.} CWI Monograph 5, Nort Holland Press, Amsterdam, 1987.
\bibitem{AndrewsHopper} B. Andrews and C. Hopper, \emph{The Ricci flow in Riemannian geometry: A complete proof of the differentiable 1/4-pinching sphere theorem}, Lecture Notes in Mathematics {\bf 2011}, Springer Heidelberg (2011), xviii-296 
\bibitem{BahuaudCC} E. Bahuaud, \emph{Ricci flow of conformally compact metrics}, Annales de l'Institut Henri Poincar\'e C, Analyse non lin\'eaire {\bf 28} no. 6 (2011), 813-835.
\bibitem{BahuaudHelliwell} E. Bahuaud and D. Helliwell, \emph{Short-time existence for some higher-order geometric flows}, Comm. PDE {\bf 36} no. 12 (2011), 2189-2207.
\bibitem{BahuaudHelliwell2} E. Bahuaud and D. Helliwell, \emph{Uniqueness for some higher order geometric flows}, Bull. London Math. Soc. {\bf 47} (2015), 980--995.
\bibitem{BGI} E. Bahuaud, C. Guenther, and J. Isenberg, \emph{Convergence Stability for Ricci flow}, J. Geom. Anal., (2020) 30:310--336.
\bibitem{Bartnik} R. Bartnik, \emph{The mass of an asymptotically flat manifold}. Comm. Pure Appl. Math. {\bf 39} (1986), no. 5., 661--693.
\bibitem{Biquard} O. Biquard, \emph{M\'etriques d'{E}instein asymptotiquement sym\'etriques}, Ast\'erisque No. 265 (2000), vi+109 pp.
\bibitem{BiquardMazzeo} O. Biquard and R. Mazzeo, \emph{A nonlinear Poisson transform for Einstein metrics on product spaces}, J. Eur. Math. Soc. {\bf 13} (2011) 1423--1475.
\bibitem{ChenChen} G. Chen and X.X. Chen, \emph{Gravitational instantons with faster than quadratic curvature decay. I.} Acta Math. {\bf 227} (2021), no. 2, 263–307.
\bibitem{ChenZhu} B.L. Chen and X.P. Zhu, \emph{Uniqueness of the Ricci flow on complete noncompact manifolds}, J. Diff. Geom. {\bf 74}, (2006), no.1., 119--154.
\bibitem{Chernoff} P. Chernoff, \emph{Essential self-adjointness of powers of generators of hyperbolic equations}, J. Func. Anal. {\bf 12}, (1973), 401--414.
\bibitem{RFIV} B. Chow, S-C Chu, D. Glickenstein, C. Guenther, J. Isenberg, T. Ivey, D. Knopf, P. Lu, F. Luo, L. Ni, \emph{The Ricci flow: techniques and applications Part IV: long time solutions and related topics}. Mathematical Surveys and Monographs {\bf 209}, AMS, Providence, RI 
\bibitem{CLN} B. Chow, P. Lu, and L. Nei, \emph{Hamilton's Ricci flow}, Graduate Studies Math., 77. AMS Science Press.
\bibitem{DegMaz} A. Degeratu and R. Mazzeo, \emph{Fredholm theory for elliptic operators on quasi-asymptotically conical spaces}, Proc. Lond. Math. Soc. {\bf 3} 116 (2018), no. 5, 1112–1160.
\bibitem{Dim-Sj} M. Dimassi and J. Sj\"ostrand, \emph{Spectral asymptotics in the semi-classical limit}, London Math. Soc. Lec. Note Series 268 (1999), Cambridge Univ. Press. 
\bibitem{Eichhorn} J. Eichhorn, \emph{The Boundedness of Connection Coefficients and their Derivatives}, Math. Nachr. 152 (1991) 145-158.
\bibitem{ES} J. Escher and G. Simonett, \emph{A center manifold analysis for the Mullins-Sekerka Model}, JDG {\bf143} (1998) 267-292
\bibitem{EMM} C.L. Epstein, R.B. Melrose, and G. Mendoza, \emph{Resolvent of the {L}aplacian on strictly pseudoconvex domains}, Acta Math. {\bf 167} (1991) 1--106.
\bibitem{FG} C. Fefferman and C. R. Graham, \emph{The Ambient metric}, Annals of Mathematics Studies, vol 179 (2012) Princeton University Press.
\bibitem{GL} C. R. Graham and J. M. Lee, \emph{Einstein metrics with prescribed conformal infinity on the ball}, Adv. Math. {\bf 87} (1991), 186--225.
\bibitem{Graham-Zworski} C. R. Graham and M. Zworski, \emph{Scattering matrix in conformal geometry}, Inv. math., vol 152 (2003), 89-–118.
\bibitem{GIK} C. Guenther, J. Isenberg, and D. Knopf, \emph{Stability of the Ricci flow at Ricci-flat metrics.} Comm. Anal. Geom. {\bf 10} (2002), no. 4, 741--777. 
\bibitem{Hel} S. Helgason, \emph{Differential geometry, Lie groups, and symmetric spaces.} Corrected reprint of the 1978 original. Graduate Studies in Mathematics, {\bf 34}. American Mathematical Society, Providence, RI, 2001. xxvi+641 pp.
\bibitem{Lee} J. M. Lee, \emph{Fredholm operators and Einstein metrics on conformally compact manifolds}, Memoirs AMS (2006) No. 864.
\bibitem{KnopfYoung} D. Knopf and A. Young, \emph{Asymptotic stability of the cross curvature flow at a hyperbolic metric}. Proc AMS. {\bf 137} (2009) No. 2, 699--709.
\bibitem{Lunardi} A. Lunardi, \emph{Analytic semigroups and optimal regularity in parabolic problems}, 2013 reprint of the 1995 original, Modern Birkh\"{a}user Classics, pp. xviii+424.
\bibitem{Martinez} A. Martinez, \emph{An introduction to semiclassical and microlocal analysis}, Universitext (2002), Springer-Verlag. 
\bibitem{MazzeoHodge} R. Mazzeo, \emph{Hodge cohomology of a conformally compact metric}, J. Diff. Geom. {\bf 28} (1988), 309--339.
\bibitem{MazzeoEdge} R. Mazzeo, \emph{Elliptic theory of Differential Edge Operators, I.}, Comm. PDE {\bf 16} (1991), no. 10, 1615--1664.
\bibitem{MazzeoPacard} R. Mazzeo and F. Pacard, \emph{Maskit combinations of {P}oincar\'e-{E}instein metrics}. Adv Math, {\bf 204} (2006) 379--412.
\bibitem{MazzeoUniCtn} R. Mazzeo, \emph{Unique continuation at infinity and embedded eigenvalues for asymptotically hyperbolic manifolds}, Amer. J. Math. {\bf 113} (1991), no. 1, 25--45.
\bibitem{McKean} H.P. McKean, \emph{An upper bound to the spectrum of $\Delta$ on a manifold of negative curvature}, J. Diff. Geom. {\bf 4} (1970), 359--366.
\bibitem{McOwen} R. McOwen, \emph{Fredholm theory of partial differential equations on complete Riemannian manifolds}, Pac. J. Math. {\bf 87} (1980), no. 1, 169–185.
\bibitem{Mel-berkeley} R. Melrose, \emph{Smooth operator algebras and K-theory}.  Lectures at UC Berkeley (2008), available at \url{https://math.mit.edu/~rbm/Bkly08/Bkly08.html}
\bibitem{Melrose-Mendoza} R. B. Melrose and G. Mendoza, \emph{Elliptic operators of totally characteristic type}, (1983) MSRI Report 047-03.
\bibitem{Petersen2016} P. Petersen, \emph{Riemannian Geometry} (3rd edition), Graduate Texts in Mathematics 171, Springer Verlag, 
\bibitem{Shubin} M. Shubin, \emph{Pseudodifferential operators and spectral theory}, (2001), Springer-Verlag
\bibitem{Yosida} K. Yosida, \emph{Functional Analysis, 6th ed} Springer (1980)
\bibitem{Zworski} M. Zworski, \emph{Semiclassical analysis} Graduate Studies in Math. 138, (2012), Amer. Math. Soc.
\end{thebibliography}
\end{document}